\documentclass[a4paper,11pt,twoside]{article}
\usepackage{ucs}
\usepackage[utf8x]{inputenc}
\usepackage[T1]{fontenc}
\usepackage{lmodern}
\usepackage{amsmath}
\usepackage{amsthm}
\usepackage{mathrsfs}
\usepackage{amsfonts}
\usepackage{array}
\usepackage[all]{xy}
\usepackage{fancyhdr}
\usepackage{graphicx}
\usepackage{epstopdf}

\author{Hubert \sc{Klaja} \thanks{Laboratoire Paul Painlev\' e,
Universit\'e Lille 1, CNRS UMR 8524, B\^at. M2,
F-59655 Villeneuve \newline d'Ascq, France ; {\tt hubert.klaja@gmail.com}}} %
\title{The numerical range and the spectrum of a product of two orthogonal projections}
\date{}

\theoremstyle{definition} \newtheorem{definition}{Definition}[section]
\theoremstyle{definition} \newtheorem{exemple}[definition]{Example}
\theoremstyle{remark}     \newtheorem{remarque}[definition]{Remark}
\theoremstyle{plain}      \newtheorem{theoreme}[definition]{Theorem}
\theoremstyle{plain}      \newtheorem{proposition}[definition]{Proposition}
\theoremstyle{plain}      \newtheorem{corollaire}[definition]{Corollary}
\theoremstyle{plain}      \newtheorem{lemme}[definition]{Lemma}
\theoremstyle{definition}

  \newcommand{\norme}[1]
    {\left\| #1 \right\|}
  \newcommand{\abs}[1]
    {\left\vert #1 \right\vert}
  \newcommand{\ps}[2]
    {\left\langle #1,#2 \right\rangle}
  
  \newcommand{\ind}[1]
    {\mathbf{1}_{#1}}
  
  \newcommand{\Range}[1]
    {\mathrm{Ran(}#1)}
  \newcommand{\Ker}[1]
    {\mathrm{Ker(}#1)}    
  \newcommand{\adherence}[1]{\overline{#1}}
      
  \newcommand{\I}{\mathfrak{i}}

  \newcommand{\B}[1]{\mathcal{B}(#1)}
  \newcommand{\N}{\mathbb{N}}
  
  \newcommand{\R}{\mathbb{R}}
  \newcommand{\C}{\mathbb{C}}
  
  \newcommand{\Ldeux}{L_2}
  
  \newcommand{\Linf}{L_\infty}
  
  \newcommand{\RaySpec}[1]{r(#1)}
  
  \newcommand{\Spec}[1]{\sigma(#1)} 
  \newcommand{\Specp}[1]{\sigma_p(#1)}
   
  \newcommand{\Span}[1]{\overline{span}\lbrace #1 \rbrace}
  \newcommand{\Sspan}[1]{span\lbrace #1 \rbrace}
    
  \newcommand{\NumRan}[1]{ W(#1) }
  \newcommand{\NumRay}[1]{ \omega(#1) }
  \newcommand{\Ellipse}[1]{\mathscr{E}(#1)}
  \newcommand{\Con}[2]{\mathrm{conv}_{#1}\lbrace #2 \rbrace }
  \newcommand{\FctionSupp}[2]{\rho_{#1}(#2)}
  \newcommand{\LdE}{L_2([0,1])}
  
  \newcommand{\Reel}[1]{\mathsf{Re}(#1)}
  
  \newcommand{\Kro}[2]{\delta_{#1,#2}}
  \newcommand{\UnitEq}{\sim}
  
  \newcommand{\F}{\mathcal{F}}

\begin{document}
\maketitle
\abstract{The aim of this paper is to describe the closure of the numerical range of the product of two orthogonal projections in Hilbert space as a closed convex hull of some explicit ellipses parametrized by points in the spectrum. Several improvements (removing the closure of the numerical range of the operator, using a parametrization after its eigenvalues) are possible under additional assumptions. An estimate of the least angular opening of a sector with vertex $1$ containing the numerical range of a product of two orthogonal projections onto two subspaces is given in terms of the cosine of the Friedrichs angle. Applications to the rate of convergence in the method of alternating projections and to the uncertainty principle in harmonic analysis are also discussed.


}

\section{Introduction}
$ \; $

{\bf Background.} The numerical range of a Hilbert space operator $T \in \B{H} $ is defined as 
$
\NumRan{T}=\lbrace \ps{Tx}{x}, x \in H, \norme{x}=1 \rbrace .
$
It is always a convex set in the complex plane (the Toeplitz-Hausdorff theorem) containing in its closure the spectrum of the operator. Also, the intersection of the closure of the numerical ranges of all the operators similar to $T$ is precisely the convex hull of the spectrum of $T$ (Hildebrandt's theorem). We refer to the book \cite{Gustafson_Rao} for these and other facts about numerical ranges. Another useful property the numerical ranges have is the following recent result of Crouzeix  \cite{Crouzeix_2007}: for every $T \in \B{H} $ and every polynomial $p$ , we have
$
\norme{p(T)}\le 12 \sup_{z \in \NumRan{T}} \abs{p(z)}
$. 

\smallskip

{\bf The problem.} The main aim of this paper is to study
the numerical range $\NumRan{T} $ and the numerical radius, defined by
$
\NumRay{T}= \sup \lbrace \abs{z}, z \in \NumRan{T} \rbrace ,
$
of a product of two orthogonal projections $T = P_{M_2}P_{M_1}$. In what follows we denote by $P_M $ the orthogonal projection onto the closed subspace $M$ of a given Hilbert space $H$. We prove a representation of the closure of $\NumRan{T} $ as a closed convex hull of some explicit ellipses parametrized by points in the spectrum $\sigma(T)$ of $T$ and we discuss several applications.
We also study the relationship between the numerical range (numerical radius) of a product of two orthogonal projections and its spectrum (resp. spectral radius). Recall that 
the spectral radius $\RaySpec{T}$ of $T \in \B{H}$ is defined as 
$\RaySpec{T}=\sup \lbrace \abs{z}, z \in \sigma(T) \rbrace  .$ 

\smallskip

{\bf Previous results.} Orthogonal projections in Hilbert space are basic objects of study in Operator theory. Products or sums of orthogonal projections, in finite or infinite dimensional Hilbert spaces, appear in various problems and in many different areas, pure or applied.
We refer the reader to a book \cite{GalantaiBook} and two recent surveys \cite{Galantai, Bottcher_Spitkovsky_2010} for more information. The fact that the numerical range of a finite product of orthogonal projections is included in some sector of the complex plane with vertex at $1$ was an essential ingredient in the proof by Delyon and Delyon \cite{Delyon_1999} of a conjecture of Burkholder, saying that the iterates of a product of conditional expectations  are almost surely convergent to some conditional expectation in an $L^2$ space (see also \cite{Crouzeix_2008,Cohen_2007}). For a product of two orthogonal projections we know that the numerical range is included in a sector with vertex one and angle $\pi/6$ (\cite{Crouzeix_2008}). 

The spectrum of a product of two orthogonal projections appears naturally in the study of the rate of convergence in the strong operator topology of 
$(P_{M_2}P_{M_1})^n$ to $P_{M_1\cap M_2}$ (cf. \cite{Deustch_2001, Bauschke_Deutsch_Hundal_2009, Deustch_Hundal_2010_I, Deustch_Hundal_2010_II, Catalin_Sophie_Muller_2010, Catalin_Sophie_Muller_2010_bis, Catalin_Lyubich_2010}). This is a particular instance of von Neumann-Halperin type theorems, sometimes called in the literature the method of alternating projections. The following dichotomy holds (see \cite{Bauschke_Deutsch_Hundal_2009}):  either the sequence $(P_{M_2}P_{M_1})^n$  converge uniformly with an exponential speed to $P_{M_1\cap M_2} $ (if $1\notin \sigma(P_{M_2}P_{M_1})$), or the sequence of alternating projections $(P_{M_2}P_{M_1})^n$ converges arbitrarily slowly in the strong operator topology (if $1\in \sigma(P_{M_2}P_{M_1})$). We refer to \cite{Catalin_Sophie_Muller_2010, Catalin_Sophie_Muller_2010_bis} for several possible meanings of ``slow convergence''.

An occurrence of the numerical range of operators related to sums of orthogonal projections appears also in some Harmonic analysis problems. The uncertainty principle in Fourier analysis is the informal assertion that a function $f \in \Ldeux(\R) $ and its Fourier transform $\mathcal{F}(f)$ cannot be too small simultaneously. Annihilating pairs and strong annihilating pairs are a way to formulate this idea (precise definitions will be given in Section 5). Characterizations of annihilating pairs and strong annihilating pairs $(S,\Sigma)$ in terms of the numerical range of the operator $P_S + \I P_\Sigma$, constructed using some associated orthogonal projections $P_S$ and $P_\Sigma$, can be found in \cite{Havin_Joricke, Lenard_1971}.

\smallskip

{\bf Main Results. } Our first contribution is an exact formula for the closure of the numerical range $ \adherence{\NumRan{P_{M_2}P_{M_1}}}$, expressed as a convex hull of some ellipses $\Ellipse{\lambda}$, parametrized by points in the spectrum ($\lambda \in \Spec{P_{M_2}P_{M_1}}$). 

\begin{definition}
  \label{DefEllipse}
  Let $\lambda \in [0,1]$. We denote $\Ellipse{\lambda}$ the domain delimited by the ellipse with foci $ 0$ and  $\lambda$, and minor axis length $\sqrt{\lambda(1-\lambda)} $.
\end{definition}

We refer to Remark \ref{RqEqEllipse} and to Figure 1 for more information about these ellipses.
%
\begin{theoreme}
  \label{CoroWP2P1}
  Let $M_1$ and $M_2$ be two closed subspaces of $H$ such that $M_1\ne H$ or $M_2\ne H$. Then the closure of the numerical range of $P_{M_2}P_{M_1} $ is the closure of the convex hull of the ellipses $\Ellipse{\lambda} $ for $\lambda \in \Spec{P_{M_2}P_{M_1}} $, i.e.:
  \[
  \adherence{\NumRan{P_{M_2}P_{M_1}}}= \adherence{\Con{}{\cup_{\lambda \in \Spec{P_{M_2}P_{M_1}}}\Ellipse{\lambda}}}.
  \]
\end{theoreme}

  The proof uses in an essential way Halmos' two subspaces theorem recalled in the next section. We will use a completely different approach to describe the numerical range (without the closure) of $T=P_{M_2}P_{M_1}$ under the additional assumption that the self-adjoint operator $T^{\ast}T=P_{M_1}P_{M_2}P_{M_1}$ is diagonalisable (see Definition \ref{DefDiag}). In this case the numerical range $\NumRan{T}$ is the convex hull of the same ellipses as before but this time parametrized by the point spectrum $\Specp{T}$ (=eigenvalues) of $T=P_{M_2}P_{M_1}$.

\begin{theoreme}
  \label{ThWP2P1Diag}
 Let $H$ be a separable Hilbert space. Let $M_1$ and $M_2$ be two closed subspaces of a Hilbert space $H$ such that $M_1 \ne H$ or $M_2 \ne H$.
  If $P_{M_1}P_{M_2}P_{M_1}$ is diagonalizable, then 
  the numerical range $\NumRan{P_{M_2}P_{M_1}}$ is the convex hull of the ellipses $\Ellipse{\lambda}$, with the $\lambda$'s being the eigenvalues of $P_{M_2}P_{M_1}$, i.e.:
  \[
  \NumRan{P_{M_2}P_{M_1}}
  = \Con{}{\cup_{\lambda \in \Specp{P_{M_2}P_{M_1}}}\Ellipse{\lambda}}.
  \] 
\end{theoreme}

Concerning the relationship between the numerical radius and the spectral radius of a product of two orthogonal projections we prove the following result.
\begin{proposition}
  \label{PropLinkRaySpecNumRan}
  Let $M_1, M_2$ be two closed subspaces of $H$. The numerical radius and the spectral radius of $P_{M_2}P_{M_1}$ are linked by the following formula:
  \[
  \NumRay{P_{M_2}P_{M_1}}= \frac{1}{2}\left( \sqrt{\RaySpec{P_{M_2}P_{M_1}}} + \RaySpec{P_{M_2}P_{M_1}} \right) .
  \]
\end{proposition}
  
  The proof is an application of Theorem \ref{CoroWP2P1} and the obtained formula is better than 
  Kittaneh's inequality \cite{Kittaneh_2003} whenever the Friedrichs angle (Definition \ref{DefCosineFriederich}) between $M_1$ and $M_2$ is positive. 
  
Theorems  \ref{CoroWP2P1} and \ref{ThWP2P1Diag} can be used to localize $\NumRan{P_{M_2}P_{M_1}}$ even if the spectrum of 
  $P_{M_2}P_{M_1}$ is unknown. We mention here the following important consequence about the inclusion of $\NumRan{P_{M_2}P_{M_1}}$ in a sector of vertex $1$ whose angular opening is expressed in terms of the cosine of the Friedrichs angle $\cos(M_1,M_2)$ between the subspaces $M_1$ and $M_2$. This is a refinement of the Crouzeix's result \cite{Crouzeix_2008} for products of two orthogonal projections.  
  
  \begin{proposition}
  \label{LemWP2P1dsSecteur}
    Let $M_1$ and $M_2$ be two closed subspaces of a Hilbert space $H$. We have the following inclusion:
  \[
  \NumRan{P_{M_2}P_{M_1}} \subset \left\lbrace z \in \C, \abs{\arg(1-z)} \le \arctan(\sqrt{\frac{\cos^2(M_1,M_2)}{4-\cos^2(M_1,M_2)}}) \right\rbrace.
  \]
\end{proposition} 
  
We next consider some inverse spectral problems and construct examples of projections such that the spectrum of their product is a prescribed compact set included in $[0,1]$. These examples will generalize to the infinite dimensional setting a result due to Nelson and Neumann \cite{Nelson_Neumann_1987}. We will also give examples that answer two open questions stated in a article of Nees \cite{Nees_1999}. 
 
  The following result allows to find $\Spec{P_{M_2}P_{M_1}}\cap[\frac{1}{4},1]$, the points of the spectrum which are larger than $1/4$, whenever the closure $\adherence{\NumRan{P_{M_2}P_{M_1}}}$ of the numerical range is known.
\begin{theoreme}
  \label{ThLienSpecWP2P1}
  Let $\alpha \in [\frac{\pi}{3},\pi]$. The following assertions are equivalent:
  \begin{enumerate}
  \item $\frac{1}{2(1-\cos(\alpha))} \in \Spec{P_{M_2}P_{M_1}} $;
  \item $\sup \lbrace \Reel{z \exp(-\I\alpha)}, z \in \NumRan{P_{M_2}P_{M_1}} \rbrace = \frac{1}{4(1-\cos(\alpha))}$.
  \end{enumerate}
\end{theoreme}
Actually it is possible to obtain a description of the entire spectrum $\Spec{P_{M_2}P_{M_1}}$ starting from $\adherence{\NumRan{P_{M_2}P_{M_1}}}$ and $\adherence{\NumRan{P_{M_2}(I-P_{M_1})}}$.

  Finally, we will explain how the relation $ 1 \in \NumRan{P_{M_2}P_{M_1}}$ is related to arbitrarily slow convergence in the von Neumann-Halperin theorem and 
 we will give new characterizations of annihilating pairs and strong annihilating pairs in terms of $\NumRan{P_SP_\Sigma} $. 

\smallskip

{\bf Organization of the paper.} The rest of the paper is organized as follows. We recall in Section 2 several preliminary notions and known facts that will be useful in the sequel. In Section 3 we discuss the results concerning the exact computation of the numerical range of a product $T$ of two orthogonal projections assuming that the spectrum, or the point spectrum, of $T$ is known. Then we will give some ``localization'' results about the numerical range of $T$ that require less informations about the spectrum of $T$.  Several examples are also given, some of them leading to an answer of two open questions from \cite{Nees_1999}. In Section 4 we discuss the inverse problem of describing the spectrum of $T$ knowing its numerical range, and the relationship between the numerical and spectral radii of $T$. The paper ends with two applications of these results, one concerning the rate of convergence in the method of alternating projections and the second one concerning the uncertainty principle.

\section{Preliminaries}
In this section we introduce some notations and recall several useful facts and results. 

\begin{definition}
  Let $E$ be a bounded subset of the complex plane $\C$. We denote by $\Con{}{E}$ the convex hull of $E$, which is the set of all convex combinations of the points in $E$, i.e.
  $$
  \Con{}{E} = \lbrace \sum_{n \in \N} x_n \varepsilon_n, \varepsilon_n \in E, x_n \in [0,1], \sum_{n \in \N} x_n = 1 \rbrace.
  $$
\end{definition}

We refer the reader to \cite{Takemoto_Uchiyama_Zsido_2003} for a proof that this definition coincides with the classical one (the smallest convex subset which contains $E$). We will also denote by $\adherence{\Con{}{E}}$ the closure of the convex hull of E. 


\subsection{Halmos' two subspaces theorem}
 For a fixed Hilbert space $H$ and a closed subspace $M$ of $H$ we denote by $M^\perp$ the orthogonal complement of $M$ in $H$ and by $P_M$ the orthogonal projection onto $M$.  Let now $M_1$ and $M_2 $ be two closed subspaces of a Hilbert space $H$. Consider the following orthogonal decomposition:
\begin{equation}
  H= (M_1 \cap M_2) \oplus (M_1 \cap M_2^\perp) \oplus (M_1^\perp \cap M_2) \oplus (M_1^\perp \cap M_2^\perp) \oplus \tilde{H},
  \label{EqDecompoPosGen}
\end{equation}
  where $\tilde{H}$ is the orthogonal complement of the first $4$ subspaces. 
  With respect to this orthogonal decomposition we can write:
\begin{align*}
  P_{M_1} &= I \oplus I \oplus 0 \oplus 0 \oplus \tilde{P_1} \\
  P_{M_2} &= I \oplus 0 \oplus I \oplus 0 \oplus \tilde{P_2} \\
  P_{M_2}P_{M_1} &= I \oplus 0 \oplus 0 \oplus 0 \oplus \tilde{P_2}\tilde{P_1} . \\ 
\end{align*}
  Suppose that the subspaces $M_1^{(\perp)} \cap M_2^{(\perp)}$ and $ \tilde{H} $ are not equal to $\lbrace 0 \rbrace$. Then using the formula $\NumRan{T\oplus S}= \Con{}{\NumRan{T},\NumRan{S}} $ (see for instance \cite{Gustafson_Rao}) we have $\NumRan{P_{M_2}P_{M_1}}= \Con{}{\lbrace 1 \rbrace \cup \lbrace 0 \rbrace \cup \NumRan{\tilde{P_2}\tilde{P_1}} } $. If $M_1 \cap M_2 = \lbrace 0 \rbrace$ and the other subspaces are not equal to $\lbrace 0 \rbrace$, then we have that  $\NumRan{P_{M_2}P_{M_1}}= \Con{}{\lbrace 0 \rbrace \cup \NumRan{\tilde{P_2}\tilde{P_1}} } $. The other cases when the others subspaces are equal to $\lbrace 0 \rbrace$ can be handle easily in the same way.

\begin{definition}
  Let $N_1, N_2$ be two closed subspaces of an Hilbert space $H$. We say that $(N_1,N_2)$ are in \emph{generic position} if:
  \[
  N_1 \cap N_2 = N_1^\perp \cap N_2 = N_1 \cap N_2^\perp = N_1^\perp \cap N_2^\perp = \lbrace 0 \rbrace.
  \]
\end{definition}

In Sections $2$ and $3$ we will denote pairs of subspaces in generic position by $(N_1, N_2)$, in order to distinguish them from pairs of general closed subspaces $(M_1, M_2)$.


We say that $A$ is unitary equivalent to $B$ (and write $A \UnitEq B$) if there exists a unitary operator $U$ such that $A=UBU^*$.
The following result, Halmos' two subspace theorem \cite{Halmos_1969}, is a useful description of orthogonal projections of two subspaces in generic position.
\begin{theoreme}
  If $(N_1,N_2)$ are in generic position, then there exists a subspace $K$ of $H$ such that $H$ is unitary equivalent to $K\oplus K$. Also, there exist two operators $C,S \in \B{K}$ such that $0\le C \le I$, $0\le S \le I$ and $C^2+S^2=I$, and such that $P_1$ and $P_2$ are simultaneously unitary equivalent to the following operators:
  \[
    P_1 \UnitEq \left(\begin{array}{cc}
I & 0 \\ 
0 & 0
\end{array} \right),
    P_2 \UnitEq \left(\begin{array}{cc}
C^2 & CS \\ 
CS  & S^2
\end{array} \right)  .   
  \]
 Moreover, there exists a self adjoint operator $T$ verifying $0 \le T \le \frac{\pi}{2} I$ such that $\cos(T)=C$ and $\sin(T)=S$.
\end{theoreme}

For a historical discussion and several applications of Halmos' two subspace theorem we refer the reader to \cite{Bottcher_Spitkovsky_2010}.

\subsection{Support functions}
The notion of support functions is classical in convex analysis.

\begin{definition}
  Let $\mathscr{S}$ be a bounded convex set in $\C$. Let $\alpha \in \R$. The support function of $\mathscr{S}$, of angle $\alpha$, is defined by the following formula:
  \[
  \FctionSupp{\mathscr{S}}{\alpha}
    = \sup \lbrace \Reel{z {\exp(-\I \alpha)}}, z \in \mathscr{S} \rbrace.
  \]
\end{definition}

The following proposition shows that the support function characterizes the closure of convex sets.

\begin{proposition}
  We denote by $\adherence{\mathscr{S}}$ the closure of $\mathscr{S}$. We have:
  $$
  \adherence{\mathscr{S}}
  = \lbrace z\in \C, \forall \alpha, \Reel{z \exp(-\I \alpha)} \le \FctionSupp{\mathscr{S}}{\alpha} \rbrace .
  $$ 
\end{proposition}

  We will need in this paper the following result about support functions.
\begin{lemme}
  \label{LemFctionSuppFusion2Ens}
  Let $\mathscr{S}_1, \mathscr{S}_2 $ be two bounded convex sets of the plane with  
  support functions $\FctionSupp{\mathscr{S}_1}{\alpha}$ and, respectively, $\FctionSupp{\mathscr{S}_2}{\alpha}$. Let $\mathscr{S} $ be such that
  $
  \FctionSupp{\mathscr{S}}{\alpha}= \max_{i=1,2} \FctionSupp{\mathscr{S}_i}{\alpha}.
  $
 Then we have $\adherence{\mathscr{S}} = \adherence{\Con{}{\mathscr{S}_1,\mathscr{S}_2}}$.
\end{lemme}
A proof of the above propositions and more information about support functions are available in \cite{Rockfellar}. 

\subsection{Cosine of Friedrichs angle of two subspaces}

  We now introduce the cosine of the Friederichs angle between two subspaces. We refer to \cite{Deustch_2001} as a source for more information.

\begin{definition}
  \label{DefCosineFriederich}
  Let $M_1,M_2 $ be two closed subspaces of $H$, with intersection $M=M_1\cap M_2$. We define the cosine of the Friederichs angle between $M_1$ and $M_2 $ by the following formula:
  $$
  \cos(M_1,M_2) = \sup \lbrace \abs{\ps{x}{y}}, x \in M_1 \cap M^\perp, y \in M_2 \cap M^\perp, \norme{x}=\norme{y}=1 \rbrace.
  $$
\end{definition}

An equivalent way (\cite{Kaylar_Weinert_1988, Deustch_2001}) to express the above cosine is given by the formula 
$\cos^2(M_1,M_2)= \norme{P_{M_1}P_{M_2}P_{M_1} - P_{M_1 \cap M_2}} $. The following result, which will be helpful later on, offers a spectral interpretation of $ \cos(M_1,M_2)$.

\begin{lemme}
\label{LCosSp}
  Let $M_1$ and $M_2 $ be two closed subspaces of $H$. Then 
  $$
  \cos(M_1,M_2) = \sup \lbrace\sqrt{\lambda} : \lambda \in \Spec{P_{M_2}P_{M_1}} \setminus \lbrace 1 \rbrace \rbrace.
  $$
\end{lemme}

This result can be seen as a consequence of Halmos' two subspace theorem (see \cite{Bottcher_Spitkovsky_2010}). We present here a different proof.

\begin{proof}
  We start by remarking that $\Spec{P_{M_2}P_{M_1}}$ is a compact subset of $[0,1]$. Indeed, we have 
  $\Spec{P_{M_1}P_{M_2}P_{M_1}} \setminus \lbrace 0 \rbrace = \Spec{(P_{M_2}P_{M_1})P_{M_1}} \setminus \lbrace 0 \rbrace = \Spec{P_{M_2}P_{M_1}} \setminus \lbrace 0 \rbrace $ and $P_{M_1}P_{M_2}P_{M_1}$ is a self-adjoint operator which is positive and of norm less or equal to one. Using the decomposition $H = (M_1 \cap M_2) \oplus  (M_1 \cap M_2)^\perp$ we can write $ P_{M_1}P_{M_2}P_{M_1} = P_{M_1 \cap M_2} \oplus (P_{M_1}P_{M_2}P_{M_1} - P_{M_1 \cap M_2})$, so we get $\Spec{P_{M_1}P_{M_2}P_{M_1}}= \Spec{P_{M_1 \cap M_2}} \cup \Spec{P_{M_1}P_{M_2}P_{M_1} - P_{M_1 \cap M_2}} $. 
Since
$$
\cos^2(M_1,M_2) = \norme{P_{M_1}P_{M_2}P_{M_1} - P_{M_1 \cap M_2}} = \sup \Spec{P_{M_1}P_{M_2}P_{M_1} - P_{M_1 \cap M_2}},$$
we obtain
$$
\cos^2(M_1,M_2) = \sup \Spec{P_{M_1}P_{M_2}P_{M_1}} \setminus \lbrace 1 \rbrace = \sup \Spec{P_{M_2}P_{M_1}} \setminus \lbrace 1 \rbrace .
$$
\end{proof}
%

\section{Description of the numerical range knowing the spectrum}
\subsection{The closure of the numerical range as a convex hull of ellipses}

  The goal of this section is to prove Theorem \ref{CoroWP2P1} using a description of the support function of $\adherence{\NumRan{P_2P_1}} $, which is a closed convex set of $\C$. This idea appeared for instance in \cite{Lenard_1971} in a different context. We will first assume that we are in generic position; the general case will be easily deduced from this particular one. The reader could see \cite{Riesz_Nagy_1990} for more details about borelian functional calculus on self adjoint operators.

\begin{lemme}
\label{RqriTilde}
Suppose that $(N_1,N_2)$ is in generic position. Denote $P_i = P_{N_i} $, $i = 1,2 $, the orthogonal projection on $N_i$. Then the support function of the numerical range of $P_2P_1$ is:
  \[
  \FctionSupp{\NumRan{P_2P_1}}{\alpha} = \sup_{\lambda \in \Spec{P_2P_1}} \frac{1}{2}(\cos(\alpha) \lambda + \sqrt{\lambda(1-\sin(\alpha)^2\lambda)}).
  \]
\end{lemme}


\begin{proof}
   We fix $\alpha \in [0,2\pi]$. We have that 
\begin{align*}
  \FctionSupp{\NumRan{P_2P_1}}{\alpha}
    &= \sup \lbrace \Reel{\ps{P_2P_1h}{h}\exp(-\I \alpha)}, h \in H, \norme{h}=1 \rbrace \\
    &= \sup \lbrace \Reel{\ps{\exp(-\I \alpha)P_2P_1h}{h}}, h \in H, \norme{h}=1 \rbrace \\
    &= \sup \lbrace \ps{\Reel{\exp(-\I \alpha)P_2P_1}h}{h}, h \in H, \norme{h}=1 \rbrace .
\end{align*}
  Applying Halmos' two subspace theorem, there exists a self adjoint operator $T$ such that
  \[
    P_2P_1 \UnitEq \left(\begin{array}{cc}
\cos(T)^2 & 0 \\ 
\cos(T)\sin(T) & 0
\end{array} \right), 
    P_1P_2 \UnitEq \left(\begin{array}{cc}
\cos(T)^2 & \cos(T)\sin(T) \\ 
0  & 0
\end{array} \right).     
  \]
  So we have that
  \[
  \Reel{\exp(-\I \alpha)P_2P_1}
    \UnitEq \left(\begin{array}{cc}
\cos(\alpha)\cos(T)^2 & \frac{\exp(\I \alpha)}{2} \cos(T)\sin(T) \\ 
\frac{\exp(-\I \alpha)}{2}\cos(T)\sin(T) & 0
\end{array} \right).
  \]
  We set
  \[
  M(t,\alpha)= \left(\begin{array}{cc}
\cos(\alpha)\cos(t)^2 & \frac{\exp(\I \alpha)}{2} \cos(t)\sin(t) \\ 
\frac{\exp(-\I \alpha)}{2}\cos(t)\sin(t) & 0
\end{array} \right).
  \]
  Then we have that $ \Reel{\exp(-\I \alpha)P_2P_1} \UnitEq M(T,\alpha)$.   
  After some computations we get that $ M(t,\alpha)= U^*(t,\alpha)D(t,\alpha)U(t,\alpha) $ with
  $$
   D(t,\alpha)= \left(\begin{array}{cc}
v_1(t,\alpha) & 0 \\ 
0 & v_2(t,\alpha)
\end{array} \right),
   U(t,\alpha)= \left(\begin{array}{cc}
\frac{2v_1(t,\alpha)}{u_1(t,\alpha)} & \frac{2v_2(t,\alpha)}{u_2(t,\alpha)} \\ 
\frac{\exp(\I \alpha)\cos(t)\sin(t)}{u_1(t,\alpha)} & \frac{\exp(\I \alpha)\cos(t)\sin(t)}{u_2(t,\alpha)}
\end{array} \right),
  $$
  and  $v_1(t,\alpha)= \frac{1}{2}(\cos(\alpha)\cos(t)^2+\cos(t)\sqrt{1-\sin(\alpha)^2\cos(t)^2})$ and $v_2(t,\alpha)= \frac{1}{2}(\cos(\alpha)\cos(t)^2-\cos(t)\sqrt{1-\sin(\alpha)^2\cos(t)^2})$ and $u_i(t,\alpha)= \sqrt{4(v_i(t,\alpha))^2 + \cos(t)^2 \sin(t)^2} $. One can easily check by passing to the limit when $t$ goes to $ \frac{\pi}{2}$ that:
  $$
  U(\frac{\pi}{2},\alpha)= \left(\begin{array}{cc}
\frac{1}{\sqrt{2}} & \frac{-1}{\sqrt{2}} \\ 
\frac{\exp(\I \alpha)}{\sqrt{2}} & \frac{\exp(\I \alpha)}{\sqrt{2}}
\end{array} \right).
  $$
   We also have that $U(t,\alpha)U^*(t,\alpha) = U^*(t,\alpha)U(t,\alpha) = I $. As all entries of  $U(t,\alpha)$ are borelians functions and $T$ is a self adjoint operator, one can define $\frac{2v_1(T,\alpha)}{u_1(T,\alpha)}$, $\frac{2v_2(T,\alpha)}{u_2(T,\alpha)}$, $\frac{\exp(\I \alpha)\cos(T)\sin(T)}{u_1(T,\alpha)}$ and $\frac{\exp(\I \alpha)\cos(T)\sin(T)}{u_2(T,\alpha)}$. So we can define $ D(T,\alpha) $ and $ U(T,\alpha) $, and we have that $M(T,\alpha)= U^*(T,\alpha)D(T,\alpha)U(T,\alpha) $ and $U(T,\alpha)U^*(T,\alpha) = U^*(T,\alpha)U(T,\alpha) = I $. So $M(T,\alpha) \UnitEq D(T,\alpha) =  v_1(T,\alpha) \oplus v_2(T,\alpha) $. Note also that $v_1(t,\alpha) \ge 0$ and $v_2(t,\alpha) \le 0$ for every $t\in [0,\frac{\pi}{2}]$ and $\alpha \in [0,2\pi]$ . As $\Spec{T} \subset [0,\frac{\pi}{2}]$, we obtain the following order relations $v_2(T,\alpha) \le 0 \le v_1(T,\alpha)$. Note that the operators $v_i(T,\alpha)$ are self-adjoint. Therefore
\begin{align*}
  \FctionSupp{\NumRan{P_2P_1}}{\alpha}
  &= \sup_{\norme{h}=1} \ps{\Reel{\exp(-\I \alpha)P_2P_1}h}{h} \\
  &= \sup_{\norme{x}=1} \ps{ v_1(T,\alpha)x}{x} \\
  &= \norme{ v_1(T,\alpha)} \\
  &= \sup_{t_0 \in \Spec{T}} v_1(t_0,\alpha).
\end{align*} 
Halmos' theorem implies that 
 \[
  P_1P_2P_1 \UnitEq \left(\begin{array}{cc}
\cos(T)^2 & 0 \\ 
0  & 0
\end{array} \right).
  \]
  We have $\Spec{P_2P_1} \setminus \lbrace 0 \rbrace = \Spec{(P_2P_1)P_1} \setminus \lbrace 0 \rbrace =\Spec{P_1P_2P_1} \setminus \lbrace 0 \rbrace $, and $\cos^2(\Spec{T}) \cup \lbrace 0 \rbrace = \Spec{P_1P_2P_1} $. Denoting $\lambda = \cos(t)^2 $ and $\tilde{v_i}(\lambda,\alpha) = \frac{1}{2}(\cos(\alpha)\lambda \pm \sqrt{\lambda(1-\sin(\alpha)^2\lambda)})$, we get 
  $ 
  \FctionSupp{\NumRan{P_2P_1}}{\alpha} = \sup_{\lambda \in \Spec{P_2P_1}} \tilde{v_1}(\lambda,\alpha) 
  $.
\end{proof}

\begin{remarque}
  \label{RqFctionSuppWP2P1}
  Using a formula due to Lumer \cite[Lemma 12]{Lumer_1961}, we obtain  \begin{align*}
  \FctionSupp{\NumRan{P_2P_1}}{\alpha}
    &= \sup \Reel{\NumRan{\exp(-\I \alpha)P_2P_1}} \\
    &= \lim_{t \rightarrow 0^+} \frac{\norme{I-t\Reel{\exp(-\I \alpha)P_2P_1}}-1}{t} \\
    &= \lim_{t \rightarrow 0^+} \frac{\norme{I-t\exp(-\I \alpha)P_2P_1}-1}{t} .
  \end{align*}
\end{remarque}

In order to make the formula of $\adherence{\NumRan{P_2P_1}}$ more explicit, we will describe it as the convex hull of ellipses $\Ellipse{\lambda} $. Recall that for $\lambda \in [0,1]$, $\Ellipse{\lambda}$ denote the domain delimited by the ellipse with foci $ 0$ and  $\lambda$, and minor axis length $\sqrt{\lambda(1-\lambda)} $.  Several of these ellipses are represented in Figure \ref{ImageEllipse}.



\begin{remarque}
  \label{RqEqEllipse}
  Other descriptions for $\Ellipse{\lambda}$ are possible. The Cartesian equation of the boundary of $\Ellipse{\lambda}$ is given by:
  \[
  \frac{(x_\lambda-\frac{\lambda}{2})^2 }{\frac{\lambda}{4}}
  + \frac{y_\lambda^2}{\frac{\lambda(1-\lambda)}{4}} = 1,
  \]
  while the parametric equation of the boundary of $\Ellipse{\lambda}$ is given by:
  
$$
  x_\lambda(t) 
    = \frac{\sqrt{\lambda}}{2}\cos(t) + \frac{\lambda}{2} , \quad
  y_\lambda(t) 
    = \frac{\sqrt{\lambda(1-\lambda)}}{2}\sin(t) .$$
\end{remarque}

\begin{figure}
\begin{center}
\includegraphics[width=0.5\textwidth ]{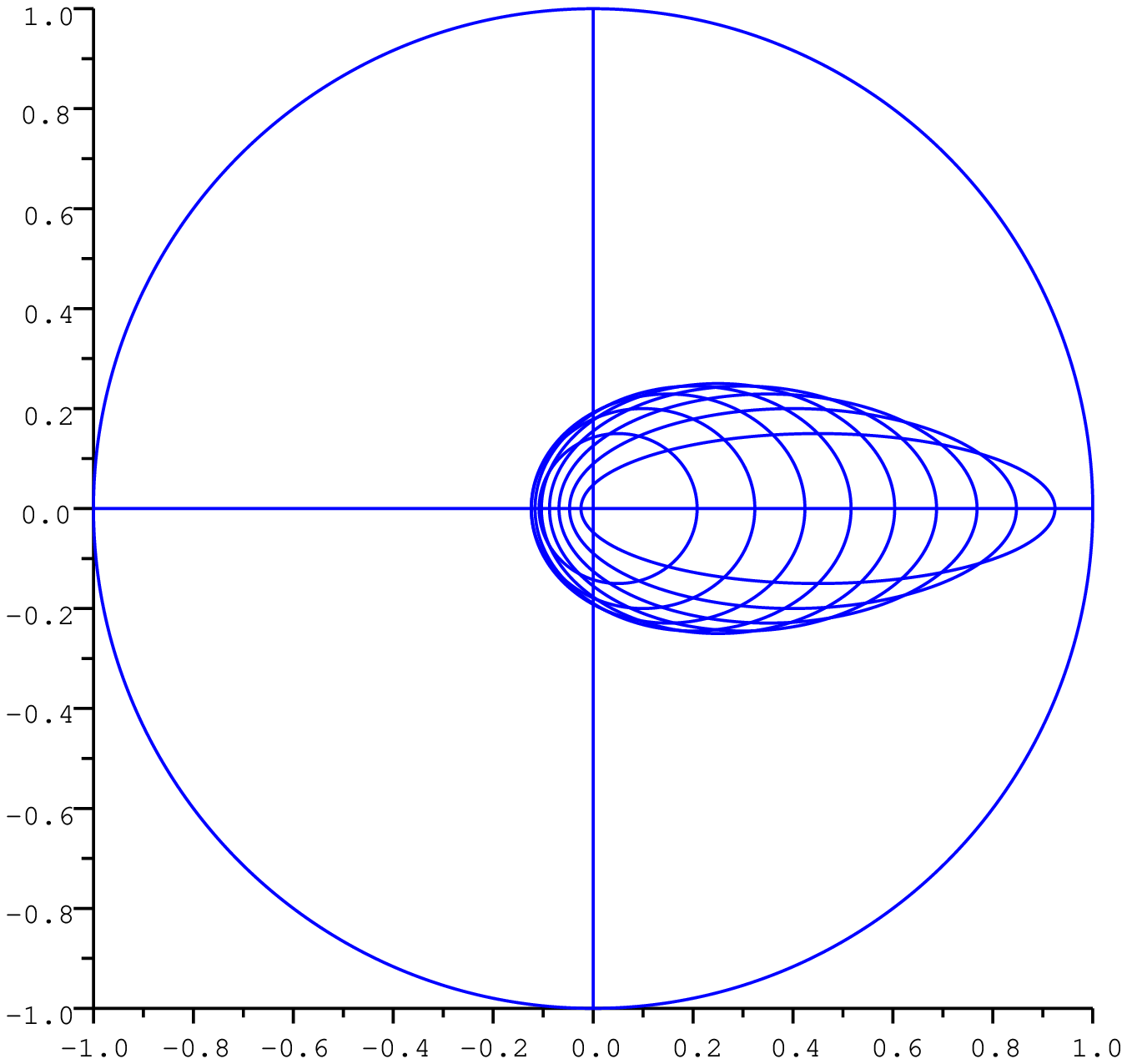}
\end{center}
\caption{Ellipse $\Ellipse{\lambda}$ for $\lambda=0.1, 0.2, \dots, 0.9$} 
\label{ImageEllipse}
\end{figure}

\begin{lemme}
  Let $\lambda \in [0,1]$. The support function of the ellipse $\Ellipse{\lambda}$ is:
  \[
  \FctionSupp{\Ellipse{\lambda}}{\alpha} = \frac{1}{2}(\cos(\alpha) \lambda + \sqrt{\lambda(1-\sin(\alpha)^2\lambda)}).
  \]
\end{lemme}

\begin{proof}

  Let $\lambda \in [0,1]$. The support function of $\Ellipse{\lambda}$ relative to the point $0$ is given by 
  $ 
  \FctionSupp{\Ellipse{\lambda}}{\alpha}
  = \sup_{t \in \R} x_\lambda(t)\cos(\alpha) + y_\lambda(t)\sin(\alpha)
  $,
where $x_\lambda(t)$ and $y_\lambda(t)$ are the parametrization of the boundary of $\Ellipse{\lambda}$.
Let $g=g_{\lambda,\alpha}$ be the function defined by the following formula:
  \begin{align*}
  g_{\lambda,\alpha}(t)
   &= x_\lambda(t)\cos(\alpha) + y_\lambda(t)\sin(\alpha) \\
   &= \frac{\lambda}{2}\cos(\alpha) + \frac{\sqrt{\lambda}}{2}\cos(\alpha)\cos(t) 
      + \frac{\sqrt{\lambda(1-\lambda)}}{2}\sin(\alpha)\sin(t).
  \end{align*}
In order to compute $\FctionSupp{\Ellipse{\lambda}}{\alpha}$ we only need to study this function for $\alpha \in [0,\pi]$ because $\Ellipse{\lambda}$ has $y=0$ as a symmetry axis.

  Suppose that $ \cos(\alpha) \ne 0 $.
  We have $g_{\lambda,\alpha}'(t_0)=0 $ if and only if $\tan(t_0) = \sqrt{1-\lambda} \tan(\alpha)$.  So the critical points of $g_{\lambda,\alpha}$ are $t_0 = \arctan(\sqrt{1-\lambda}\tan(\alpha))$ and $t_1 = \arctan(\sqrt{1-\lambda}\tan(\alpha))+ \pi$. We denote $\epsilon_0=1, \epsilon_1=-1$. Using standard trigonometric identities, we get
%
  
 \[
   \cos(t_i) = \epsilon_i \frac{1}{\sqrt{1+(1-\lambda)\tan(\alpha)^2}},
   \sin(t_i) = \epsilon_i \frac{\sqrt{1-\lambda}\tan(\alpha)}{\sqrt{1+(1-\lambda)\tan(\alpha)^2}}.
 \]
 We denote $\epsilon_\alpha=\frac{\cos(\alpha)}{\abs{\cos(\alpha)}}$. Using again some trigonometry formulas, we have:
\begin{align*}
  2g_{\lambda,\alpha}(t_i)
   =& \lambda\cos(\alpha) + \epsilon_i\sqrt{\lambda}\cos(\alpha)\frac{1}{\sqrt{1+(1-\lambda)\tan(\alpha)^2}} \\
    &+ \epsilon_i\sqrt{\lambda(1-\lambda)}\sin(\alpha)\frac{\sqrt{1-\lambda}\tan(\alpha)}{\sqrt{1+(1-\lambda)\tan(\alpha)^2}} \\
   =&  \lambda\cos(\alpha) + \epsilon_i\epsilon_\alpha \sqrt{\lambda}\sqrt{1-\lambda\sin(\alpha)^2}.
\end{align*}
We finally obtain that
  \[
  \FctionSupp{\Ellipse{\lambda}}{\alpha} 
  = \frac{1}{2}\left(\lambda\cos(\alpha) + \sqrt{\lambda}\sqrt{1-\lambda\sin(\alpha)^2}\right).
  \]
  
  Suppose now that $\cos(\alpha) = 0$. Then $  g_{\lambda,\alpha}(t) = \frac{\sqrt{\lambda(1-\lambda)}}{2}\sin(t) $. So we get that in all situations
 $
 \FctionSupp{\Ellipse{\lambda}}{\alpha} = \frac{\sqrt{\lambda(1-\lambda)}}{2}
 $.
 We obtain  $
  \FctionSupp{\Ellipse{\lambda}}{\alpha} =  \frac{1}{2}(\cos(\alpha)\lambda + \sqrt{\lambda(1-\sin(\alpha)^2\lambda)})
  $ for every $\alpha$.
\end{proof}

Now we can easily prove Theorem \ref{CoroWP2P1} in the "generic position" case. 
\begin{theoreme}
  \label{ThWP2P1PosGen}
  If $(N_1,N_2)$ are in generic position, then:
  \[
  \adherence{\NumRan{P_2P_1}}= \adherence{\Con{}{\cup_{\lambda \in \Spec{P_2P_1}}\Ellipse{\lambda}}}.
  \]
\end{theoreme}

\begin{proof}
  We first notice that:
  $$
    \FctionSupp{\NumRan{P_2P_1}}{\alpha} 
    = \sup_{\lambda \in \Spec{P_2P_1}}  \frac{1}{2}(\cos(\alpha)\lambda \pm \sqrt{\lambda(1-\sin(\alpha)^2\lambda)}) 
    = \sup_{\lambda \in \Spec{P_2P_1}} \FctionSupp{\Ellipse{\lambda}}{\alpha}
  .$$
  As the support function characterizes the closure of a convex bounded set, we simply use Lemma \ref{LemFctionSuppFusion2Ens} to conclude.
\end{proof}

  The proof of the general case follows now by combining the previous theorem with the decomposition (\ref{EqDecompoPosGen}). 
  



\begin{proof}[Proof of Theorem \ref{CoroWP2P1}]
  \label{RqPasPosGen}
  Recall that $M_1\ne H$ or $M_2 \ne H$. We use the notation of the orthogonal decomposition (\ref{EqDecompoPosGen}) of $H$. Suppose that $\tilde{H} =\lbrace 0 \rbrace$. Then $P_{M_2}P_{M_1}$ is the direct sum of $0$ and $I$ (or is zero if $M_1 \cap M_2 = \lbrace 0 \rbrace $). Then it is easy to see that $\Ellipse{0} = \lbrace 0 \rbrace$ and $\Ellipse{1} = [0,1]$. So we have $\NumRan{P_{M_2}P_{M_1}}= [0,1] = \Con{}{\Ellipse{0} \cup \Ellipse{1}} $. When $M_1 \cap M_2 = \lbrace 0 \rbrace $, we have $\NumRan{P_{M_2}P_{M_1}}= \lbrace 0 \rbrace = \Ellipse{0} $.

  Suppose $\tilde{H} \ne \lbrace 0 \rbrace$, and $M_1^\perp \cap M_2^\perp \ne \lbrace 0 \rbrace $ (the cases $M_1^\perp \cap M_2 \ne \lbrace 0 \rbrace $ and $M_1\cap M_2^\perp \ne \lbrace 0 \rbrace $ are similar). On the space $M_1^\perp \cap M_2^\perp$, we have $P_{M_2}P_{M_1}=0$. The numerical range of $P_{M_2}P_{M_1}$ on $(M_1^\perp \cap M_2^\perp) \oplus \tilde{H}$ is
  $  
  \Con{}{\lbrace 0 \rbrace \cup \adherence{\Con{}{\cup_{\lambda \in \Spec{P_2P_1}}\Ellipse{\lambda}}}}
  $.
  As $ \Ellipse{0} = \lbrace 0 \rbrace \subset \Ellipse{\lambda} $ for all $\lambda \in [0,1]$, the numerical range of $P_{M_2}P_{M_1}$ on $(M_1^\perp \cap M_2^\perp) \oplus \tilde{H}$ is given by  
  $
  \adherence{\Con{}{\cup_{\lambda \in \Spec{P_2P_1}}\Ellipse{\lambda}}}
  $.
    
  Suppose $M_1 \cap M_2 \ne \lbrace 0 \rbrace $. As $P_{M_2}P_{M_1}=I$ on the intersection $M_1 \cap M_2$, the numerical range of $P_{M_2}P_{M_1}$ on $(M_1 \cap M_2) \oplus \tilde{H}$ is
  $ 
  \Con{}{\lbrace 1 \rbrace \cup \adherence{\Con{}{\cup_{\lambda \in \Spec{P_2P_1}}\Ellipse{\lambda}}}}
  $.
  For every $\lambda \in [0,1]$ we have $0 \in \Ellipse{\lambda}$. As $\tilde{H}\neq \lbrace 0 \rbrace$, the numerical range of $P_{M_2}P_{M_1}$ on $(M_1 \cap M_2) \oplus \tilde{H}$ is
  $ 
  \Con{}{[0,1] \cup \adherence{\Con{}{\cup_{\lambda \in \Spec{P_2P_1}} \Ellipse{\lambda}}}}
  $.
  But $ \Ellipse{1}=[0,1] $. So, finally, the numerical range of $P_{M_2}P_{M_1}$ on $(M_1 \cap M_2) \oplus \tilde{H}$ is 
  $
  \adherence{\Con{}{\cup_{\lambda \in \Spec{P_{M_2}P_{M_1}} }\Ellipse{\lambda}}}
  $.
  This proves the theorem.
\end{proof} 
  In the case when $P_{M_1} = I$ and $P_{M_2} = I$, we have of course that $\NumRan{P_{M_2}P_{M_1}}= \lbrace 1 \rbrace$. 

\begin{remarque}
  \label{RqWIdempotent}
  In \cite{Corach_Maestripieri_2011}, Corach and Maestripieri proved that the Moore-Penrose pseudoinverse of a product of two orthogonal projections is idempotent (possibly unbounded). Conversely, the Moore-Penrose pseudoinverse of an idempotent is a product of two orthogonal projections. It is well known that the numerical range of a (bounded) idempotent is an ellipse (see \cite{Simoncini_Szyld_2010}). By using Halmos' theorem in a similar way as before, it is possible to prove that the closure of the numerical range of an idempotent $E$ is the convex hull of the domains delimited by the ellipses $\mathcal{E}^+(\lambda)$ of foci $0,1$ and of minor axis length $\sqrt{\frac{1-\lambda}{\lambda}}$, for $ \lambda$ describing the spectrum $\Spec{E^+}$ of the Moore-Penrose pseudoinverse $E^+$ of $E$, i.e.:
$$
  \adherence{\NumRan{E}}= \Con{}{\cup_{\lambda \in \Spec{E^+}}\mathcal{E}^+(\lambda)}.
$$
  As $\mathcal{E}^+(\lambda_1) \subset \mathcal{E}^+(\lambda_2) $, if $\lambda_1 \le \lambda_2 $, the convex hull of all these ellipses will be just the biggest one, and we find another proof that $\NumRan{E}$ is an ellipse.
\end{remarque}

\subsection{$\NumRan{P_2P_1}$  when $P_1P_2P_1$ is diagonalizable}

Let $(N_1, N_2)$ be a pair of closed subspaces of $H$. Denote $P_i = P_{N_{i}}$.
Suppose that $(N_1, N_2) $ is in generic position. 
As we have seen in the proof of Theorem \ref{CoroWP2P1}, if we get $\NumRan{P_2P_1} $ when $(N_1, N_2) $ is in generic position, we can manage to get $\NumRan{P_2P_1} $ in the general case.

In this section we always assume for simplification that $H$ is separable and make the hypothesis that $P_1P_2P_1$ is diagonalizable, according to the following definition.
\begin{definition}
  \label{DefDiag}
We say that $P_1P_2P_1$ is \emph{diagonalizable} if there exists an orthonormal basis 
$\left( \tilde{h}_n \right)_{n\in \N}$ of  $H$ and a sequence of scalars $\left(\tilde{\lambda}_n\right)_{n \in \N}$ such that:
  \[
  P_1P_2P_1x = \sum_{n\in \N} \tilde{\lambda}_n \ps{x}{\tilde{h}_n}\tilde{h}_n  \quad (x \in H).
  \]
\end{definition}
This happens for instance when $P_2P_1$ is a compact operator. Using our diagonalizability assumption, it will be possible to decompose $P_2P_1$ as a direct sum of $2\times 2$ matrices. As we know that the numerical range of such a matrix is an ellipse, this will permit to deduce the numerical range of $P_2P_1$. 
  We first notice that $0 \le P_1P_2P_1 \le I$. Therefore $ 0\le \tilde{\lambda}_n \le 1$. The next lemma characterizes when $ \tilde{h}_n \in N_1 $.

\begin{lemme}
  Suppose that $(N_1,N_2)$ is in generic position. We have:
  \begin{enumerate}
	\item $ \tilde{h}_n \in N_1 \Leftrightarrow \tilde{\lambda}_n \neq 0 $
	\item $ \tilde{h}_n \in N_1^\perp \Leftrightarrow \tilde{\lambda}_n = 0 $.
\end{enumerate}
\end{lemme}

\begin{proof}
  We know that $ P_1P_2P_1 \tilde{h}_n = \tilde{\lambda}_n\tilde{h}_n $. If $\tilde{\lambda}_n \neq 0$, then $ \tilde{h}_n = \frac{1}{\tilde{\lambda}_n}P_1P_2P_1\tilde{h}_n \in N_1 $.  
  If $\tilde{\lambda}_n = 0$, then  $P_1P_2P_1\tilde{h}_n=0$. So $ P_2P_1\tilde{h}_n \in N_1^\perp \cap N_2 =\lbrace 0 \rbrace $, because we are in generic position. So $P_2P_1\tilde{h}_n=0$. We get $ P_1\tilde{h}_n \in N_2^\perp \cap N_1 =\lbrace 0 \rbrace $, $P_1\tilde{h}_n=0$ and thus $\tilde{h}_n \in N_1^\perp$.
\end{proof}

  From now on, we just need those vectors $\tilde{h}_n$ which are in $N_1$.  For simplification, we denote these vectors as $(h_n)_{n\in \N}$, each one correspond to a nonzero $ \lambda_n$. This means that $ P_1P_2P_1 h_n = \lambda_n h_n $. As we have $h_n\in N_1$, we get $P_1 h_n = h_n$.
  We denote (see Figure \ref{FigDiag}) 
  $$ w_n = P_2 h_n  ,  \tilde{w}_n=\frac{w_n}{\norme{w_n}}, f_n = (I-P_1)P_2 h_n ,  \tilde{f}_n=\frac{f_n}{\norme{f_n}} .$$
  
\begin{figure}  
\begin{center}
\begin{picture}(100.055,100)(0,0)
\put(0,0){\vector(1,0){100}}
\put(0,0){\vector(0,1){100}}
\put(0,0){\vector(1,1){70.710678118654752440084436210485}}
\put(0,0){\vector(0,1){50}}
\put(0,0){\vector(1,1){50}}

\multiput(0,50)(20,0){3}{\line(1,0){10}}
\multiput(100,0)(-20,20){3}{\line(-1,1){10}}

\put(95,-10){$h_n$}
\put(-15,95){$\tilde{f}_n$}
\put(-15,45){$f_n$}
\put(80,80){$\tilde{w}_n$}
\put(35,55){$w_n$}

\qbezier(13,13)(17.5,8)(20,0)
\put(20,10){$\theta_n$}
\end{picture}
\end{center}
\caption{\label{FigDiag} }
\end{figure} 
  
\begin{lemme}
  \label{LemPs}
  We have $ \ps{w_n}{w_k}= \Kro{n}{k}\lambda_n $ and $ \ps{w_n}{h_k}= \Kro{n}{k}\lambda_n $, where $\Kro{n}{k} $ is the Kronecker symbol, whose value is 1 if $n=k$, and $0$ otherwise.
\end{lemme}

\begin{proof}
  For the first equality, we have that
  $
    \ps{w_n}{w_k}
      = \ps{P_2P_1h_n}{P_2P_1h_k} 
      = \ps{P_1P_2P_1h_n}{h_k} 
      = \lambda_n \ps{h_n}{h_k} 
      = \Kro{n}{k}\lambda_n
  $.
  For the other one, we have
  $
    \ps{w_n}{h_k}
      = \ps{P_2P_1h_n}{P_1h_k} 
      = \ps{P_1P_2P_1h_n}{h_k} 
      = \lambda_n \ps{h_n}{h_k} 
      = \Kro{n}{k}\lambda_n
  $.  
\end{proof}

\begin{corollaire}
Let $ \Sspan{h,w} $ be the closed subspace of $H$ generated by $h$ and $w$. If $n \neq k $, then $ \Sspan{h_n,w_n}$  is orthogonal to $ \Sspan{h_k,w_k}$.
\end{corollaire}

\begin{proposition}
The range of  $ \Sspan{h_n,w_n}$ by $P_2P_1$ verifies
  \[
  P_2P_1(\Sspan{h_n,w_n}) =  \Sspan{w_n} \subset \Sspan{h_n,w_n} .
  \]
\end{proposition}

\begin{proof}
  We just need to prove that $P_2P_1(h_n)$ and $P_2P_1(w_n)$ are collinear with $w_n$. We have $P_2P_1(h_n)=w_n$. As $h_n$ is an eigenvector of $P_1P_2P_1$, we obtain
$   P_2P_1(w_n) 
      = P_2P_1P_2P_1(h_n) 
      = P_2(\lambda_n h_n) 
      = \lambda_n w_n
$.
\end{proof}

\begin{lemme}
  We have $\Sspan{h_n,w_n} = \Sspan{h_n,f_n} $.
\end{lemme}

\begin{proof}
As both of them are subspaces of dimension $2$, it will be enough to show that $\Sspan{h_n,w_n} \subset \Sspan{h_n,f_n} $. As $h_n \in \Sspan{h_n,f_n} $, we just need to prove that $w_n \in  \Sspan{h_n,f_n}$. We have
$
    w_n
    = P_2P_1 h_n 
    = P_1P_2P_1 h_n + (I-P_1)P_2P_1 h_n 
    = \lambda_n h_n + f_n
$.
  So $w_n \in  \Sspan{h_n,f_n}$.
\end{proof}

\begin{corollaire}
If $n \ne k$, then $\Sspan{h_n,f_n}$ is orthogonal to $\Sspan{h_k,f_k}$. Moreover,
  \[
  P_2P_1(\Sspan{h_n,f_n}) =  \Sspan{w_n} \subset \Sspan{h_n,f_n}.
  \]
\end{corollaire}

\begin{proposition}
  We have $\adherence{P_2(N_1)}=N_2$.
\end{proposition}

\begin{proof}
  The inclusion $\adherence{P_2(N_1)} \subset N_2$ is obvious. In order to prove that $\adherence{P_2(N_1)} \supset N_2$, it is enough to show that $P_2(N_1)^\perp \subset N_2^\perp $. 
  Let $ y \in P_2(N_1)^\perp$. Then, for every $x \in N_1$, we have
$
  0 = \ps{y}{P_2(x)} = \ps{P_2(y)}{x}
$.
  So $ P_2(y) \in N_1^\perp $. As $ P_2(y) \in N_2 $ and $ N_1^\perp \cap N_2 = \lbrace 0 \rbrace $, we obtain $ P_2(y)=0$. So $ y \in N_2^\perp $. 
\end{proof}

\begin{corollaire}
  The vectors $(\tilde{w}_n)_{n \in \N} $ forms an orthonormal basis of $N_2$.
\end{corollaire}

\begin{proof}
  We know from Lemma \ref{LemPs} that $(\tilde{w}_n)_{n \in \N} $ is an orthonormal system in $N_2$. It remains to show that it is a generating system.
We notice that the inclusion $ P_2(N_1) \subset \Span{w_n, n \in \N} $ implies, using $ \adherence{P_2(N_1)} = N_2 $ and $ \Span{w_n, n \in \N} = \Span{\tilde{w}_n, n \in \N} \subset N_2 $, that
  $$
  N_2 = \adherence{P_2(N_1)} \subset \Span{\tilde{w}_n, n \in \N} \subset N_2 ,
  $$
  and then $ N_2 = \Span{\tilde{w}_n, n \in \N}$.  
  Let us show that $ P_2(N_1) \subset \Span{w_n, n \in \N} $. For $ x \in N_1$, there exists a sequence $(\nu_n)$ such that $ x = \sum_n \nu_n h_n $. Therefore 
  $
    P_2(x)
    = P_2(\sum_n \nu_n h_n) 
    = \sum_n \nu_n P_2(h_n) 
    = \sum_n \nu_n w_n     
  $.
 Finally $ P_2(x) \in \Span{w_n, n \in \N} $.
\end{proof}

Similarly, we can also show the following proposition.

\begin{proposition}
  We have $\adherence{(I-P_1)(N_2)}=N_1^\perp$.  
  Moreover, $(\tilde{f}_n)_{n \in \N} $ is an orthonormal basis of $N_1^\perp$.
\end{proposition}

\begin{corollaire}
  \label{CoroDecP2P1Diag}
  The operator $P_2P_1$ can be written as a direct sum of $2\times2$ matrices, i.e.:
  \[
  P_2P_1 = \bigoplus_{n \in \N} P_2P_1\mid_{\Sspan{h_n, \tilde{f}_n}} .
  \]
\end{corollaire}

\begin{proof}
  As $ \tilde{f}_n= \frac{f_n}{\norme{f_n}} $, we have $\Sspan{h_n,f_n} = \Sspan{h_n, \tilde{f}_n } $, and 
  $
  P_2P_1(\Sspan{h_n,\tilde{f}_n}) \subset \Sspan{h_n,\tilde{f}_n} .
  $
Also, $\Sspan{h_n,\tilde{f}_n}$ is orthogonal to $\Sspan{h_k,\tilde{f_k}}$ whenever $n \ne k$. Moreover, $(\tilde{f}_n)_{n \in \N} $ is an orthonormal basis of $N_1^\perp$.  
  We can write $H$ as $   H
  = N_1 \oplus N_1^\perp = \Span{h_n, n \in \N} \oplus \Span{\tilde{f}_n, n \in \N} 
  = \oplus_n \Sspan{h_n,\tilde{f}_n}$
%
which proves the result.
\end{proof}

\begin{lemme}
  \label{Lem2x2}
  With respect to the orthonormal basis $(h_n, \tilde{f}_n)$, the restriction of $P_2P_1$ to its invariant subspaces $\Sspan{h_n,\tilde{f}_n} $ is given by:
  \[
  P_2P_1\mid_{\Sspan{h_n, \tilde{f}_n}}
  =
  \left( \begin{array}{cc}
\lambda_n & 0 \\ 
\sqrt{\lambda_n(1-\lambda_n)} & 0
\end{array} \right).
  \]
\end{lemme}

\begin{proof}
  As $\tilde{f}_n \in N_1^\perp$, we have $P_1\tilde{f}_n=0$, so $P_2P_1\tilde{f}_n=0$. We can represent $P_2P_1h_n$ as:
  $
  P_2P_1h_n
  = P_1P_2P_1h_n + (I-P_1)P_2P_1h_n 
  = \lambda_n h_n + f_n 
  = \lambda_n h_n + \norme{f_n}\tilde{f}_n 
  $.
  
  In order to complete the proof, we have to show that $\norme{f_n}=\sqrt{\lambda_n(1-\lambda_n)}$. We have 
  $
  \norme{f_n}^2
  = \norme{(I-P_1)P_2P_1h_n}^2 
  = \norme{P_2P_1h_n}^2 - \norme{P_1P_2P_1h_n}^2 
  = \ps{P_1P_2P_1 h_n}{h_n} - \norme{\lambda_n h_n}^2 
  = \lambda_n - \lambda_n^2 
  $ .
\end{proof}

\begin{remarque}
  As $ 0 \le P_1P_2P_1 \le I$, we have $0 \le \lambda_n \le 1$ for every $n$. There exists $\theta_n $ such that $0 \le \theta_n \le \frac{\pi}{2}$ and $\cos(\theta_n)^2=\lambda_n$. Now we can rewrite $P_2P_1\mid_{\Span{h_n, \tilde{f}_n}}$ as:
  \[
  P_2P_1\mid_{\Sspan{h_n, \tilde{f}_n}}
  =
  \left( \begin{array}{cc}
\cos(\theta_n)^2 & 0 \\ 
\cos(\theta_n)\sin(\theta_n) & 0
\end{array} \right).
  \]
  This corresponds to the matrix of the composition of two orthogonal projections in the plane, projecting onto two lines of angle $\theta_n$.
\end{remarque}

\begin{corollaire}
\label{CoroEllipseP2P1Diag}
The numerical range $\NumRan{P_2P_1\mid_{\Sspan{h_n, \tilde{f}_n}}} $  is the ellipse $\Ellipse{\lambda_n} $.
\end{corollaire}

\begin{proof}
  This is consequence of the classical ellipse lemma for the numerical range of a $2\times 2$ matrix (see for instance \cite{Gustafson_Rao}).
\end{proof}

The following corollary is a ``generic position" version of Theorem \ref{ThWP2P1Diag}.

\begin{corollaire}
  \label{CoroWP2P1DiagPosGen}
  Let $(N_1,N_2)$ be two subpsaces in generic position such that $P_1P_2P_1$ is diagonalizable, then 
  the numerical range $\NumRan{P_2P_1}$ is the convex hull of the ellipses $\Ellipse{\lambda}$ for all the $\lambda$'s which are non zero eigenvalues of $P_2P_1$, i.e.:
  \begin{align*}
  \NumRan{P_2P_1}
  &= \Con{}{\cup_{\lambda \in \Specp{P_2P_1} \setminus \lbrace 0\rbrace} \Ellipse{\lambda}} .
  \end{align*}
\end{corollaire}

\begin{proof}
  From Corollary \ref{CoroDecP2P1Diag}, we have that $H = \oplus_{n \in \N} \Sspan{h_n, \tilde{f}_n} $. Let $x = \oplus_{n \in \N} x_n $ be a vector in $H$ such that $x_n \in \Sspan{h_n, \tilde{f}_n} $ and $\norme{x}^2 = \sum_{n \in \N} \norme{x_n}^2 =1 $. Then $\ps{P_2P_1x}{x} = \sum_{n\in \N}\ps{P_2P_1x_n}{x_n} = \sum_{n\in \N} \norme{x_n}^2\frac{\ps{P_2P_1x_n}{x_n}}{\norme{x_n}^2} $. From Corollary \ref{CoroEllipseP2P1Diag}, we have that $\frac{\ps{P_2P_1x_n}{x_n}}{\norme{x_n}^2} \in \Ellipse{\lambda_n} $. So $\NumRan{P_2P_1} \subset \Con{}{\cup_{\lambda \in \Specp{P_2P_1} \setminus \lbrace 0\rbrace} \Ellipse{\lambda}} $.
  
  Let $(\alpha_n)_{n \in \N}$ be a sequence such that $\alpha_n \in [0,1] $ and $\sum_{n \in \N} \alpha_n =1 $. Let $(\epsilon_n)_{n \in \N}$ be a sequence such that $\epsilon_n \in \Ellipse{\lambda_n} $. From Corollary \ref{CoroEllipseP2P1Diag}, there exist some $x_n \in \Sspan{h_n, \tilde{f}_n} $ such that $\norme{x_n}=1 $ and $\epsilon_n = \ps{P_2P_1x_n}{x_n} $. Let $x=\sum_{n \in \N} \alpha_n x_n$, then $\ps{P_2P_1x}{x} = \sum_{n \in \N} \alpha_n \epsilon_n $. So $\Con{}{\cup_{\lambda \in \Specp{P_2P_1} \setminus \lbrace 0\rbrace} \Ellipse{\lambda}} \subset \NumRan{P_2P_1} $.
\end{proof}

  Using the same idea as in the proof of Theorem \ref{CoroWP2P1}, we can deduce Theorem \ref{ThWP2P1Diag} from Corollary \ref{CoroWP2P1DiagPosGen}.

  With this Corollary, we can see that the numerical range of a product of two orthogonal projections is not closed in general. 

\begin{exemple}
  Let $(N_1,N_2) $ be two subspaces in generic position and denote $P_{N_i}=P_i $. Suppose that $P_1P_2P_1$ is diagonalizable. Moreover suppose that there exists an orthonormal basis $(h_n)_{n \in \N^*} $ of $N_1$ such that for all $x \in H$ we have
  $$
  P_1P_2P_1 x = \sum_{n \in \N^*} \left( 1-\frac{1}{n+1} \right) \ps{x}{h_n}h_n.
  $$  
Then we have that $\Specp{P_2P_1} = \lbrace 1-\frac{1}{n+1}, n \in \N^* \rbrace \cup \lbrace 0 \rbrace $ and $\Spec{P_2P_1} = \Specp{P_2P_1} \cup \lbrace 1 \rbrace $.  Therefore by Corollary \ref{CoroWP2P1DiagPosGen} and Theorem \ref{CoroWP2P1}, we have that $ \NumRan{P_2P_1} = \Con{}{\cup_{\lambda \in \Specp{P_2P_1} \setminus \lbrace 0\rbrace} \Ellipse{\lambda}} $ and $ \adherence{\NumRan{P_2P_1}}= \adherence{\Con{}{\cup_{\lambda \in \Spec{P_{M_2}P_{M_1}}}\Ellipse{\lambda}}} $.

  We have that $1 \in  \adherence{\NumRan{P_2P_1}} $ but $1 \notin \NumRan{P_2P_1} $. Note that $1 \in \Ellipse{\lambda} $ if and only if $\lambda = 1 $. We have that (see Remark \ref{RqEqEllipse})
  $$
  x_{1-\frac{1}{n+1}}(0) 
    = \frac{1}{2}(\sqrt{1 - \frac{1}{n+1}} + 1-\frac{1}{n+1} )
    \in \Ellipse{1-\frac{1}{n+1}}
    \subset \NumRan{P_2P_1}.
  $$
As $\lim_{n \rightarrow \infty} x_{1-\frac{1}{n+1}}(0)  = 1 $ we have that $1 \in  \adherence{\NumRan{P_2P_1}} $ .

 Suppose that $1 \in \NumRan{P_2P_1} $. Then there exists $x \in H$ such that $\norme{x}= 1 $ and $ \ps{P_2P_1x}{x} = 1 $. As $ 1 = \abs{\ps{P_2P_1x}{x}} \le \norme{P_2P_1x}\norme{x} \le 1 $, we have that $ \abs{\ps{P_2P_1x}{x}} = \norme{P_2P_1x}\norme{x}$, so there exists $\lambda$ such that $P_2P_1x = \lambda x $. We get that $ 1 = \ps{P_2P_1x}{x} =\lambda \ps{x}{x} = \lambda $. So $\lambda = 1 \in \Specp{P_2P_1} $. This is a contradiction with $1 \notin \Specp{P_2P_1} $, so $1 \notin \NumRan{P_2P_1} $.
\end{exemple}

\begin{exemple}
  \label{RqExNonDiag}
  There are non-trivial examples where $P_1P_2P_1$ admits only $0$ as eigenvalue (hence $P_1P_2P_1$ is not diagonalizable). Let $T \in \B{\LdE} $ be defined by $Tf(x) = x f(x)$. One can easily show that $T$ is an injective positive contraction that has no eigenvalues, with $\Ker{I-T}=\lbrace 0 \rbrace$ and $\Spec{T}=[0,1]$.
   If we set $C=T^{1/2}$ and $S=(I-T)^{1/2}$, we easily see that $C$ and $S$ are injective and positive contractions with no eigenvalues such that $C^2+S^2=I $. Moreover $C$ and $S$ commute. We set $ H = \LdE \oplus \LdE $ and
\begin{align*}
  P_1
  =
  \left( \begin{array}{cc}
I & 0 \\ 
0 & 0
\end{array} \right), \, 
  &P_2
  =
  \left( \begin{array}{cc}
C^2 & CS \\ 
CS  & S^2
\end{array} \right) .
\end{align*}
  Then $P_1$ and $P_2$ are orthogonal projections onto subspaces sitting in generic position, and
  \[
  P_1P_2P_1
  =
  \left( \begin{array}{cc}
C^2 & 0 \\ 
0   & 0
\end{array} \right) .
  \]
  Suppose there exist $f\oplus g \in H$ and $\lambda \in \Spec{P_2P_1}$ such that $P_1P_2P_1(f\oplus g) = \lambda (f\oplus g)$. Then $xf(x) = \lambda f(x)$ almost everywhere, and $0= \lambda g(x)$. This implies that $\lambda = 0$ and $f = 0$. So $0$ is the only eigenvalue of $P_1P_2P_1$. However, we have $\Spec{P_1P_2P_1}= \Spec{T} \cup \lbrace 0 \rbrace = [0,1] $.
\end{exemple}

\begin{remarque}
  At the end of \cite{Nees_1999}, the author asks if $\norme{P_2P_1}^2$ is an accumulation point of eigenvalues, and if the spectrum $P_2P_1$ without zero consists only of eigenvalues. The previous example answers these two questions negatively. 
\end{remarque}

\subsection{Localization of $\NumRan{P_2P_1}$}

First we have this simple consequence of Theorem \ref{CoroWP2P1}.
\begin{corollaire}
  \label{CoroWP2P1Subset}
  Let $P_1,P_2$ be two orthogonal projections. We have:
  \[
  \adherence{\NumRan{P_2P_1}} \subset \adherence{\Con{\lambda \in [0,1]}{\Ellipse{\lambda}}}.
  \]
\end{corollaire}
\begin{proof}
  If $P_1=P_2=I$ this is clear since $\NumRan{I} = \lbrace 1 \rbrace$. Now suppose that $P_1 \ne I$ or $P_2 \ne I$. We use Theorem \ref{CoroWP2P1} and the fact that $\Spec{P_2P_1} \subset [0,1]$, so we have the 
  inclusion $\Con{\lambda \in \Spec{P_{M_2}P_{M_1}}}{\Ellipse{\lambda}} \subset \Con{\lambda \in [0,1]}{\Ellipse{\lambda}}$. 
\end{proof}
   This corollary says that if we can include $\adherence{\Con{\lambda \in [0,1]}{\Ellipse{\lambda}}}$ (see Figure \ref{FigConElambda}) in a subset of $\C$, then for any pair of projection $P_1, P_2$ we can include $\NumRan{P_2P_1}$ in the same subset. The next lemma is an example of localization of the numerical range using Corollary \ref{CoroWP2P1Subset}.  

\begin{figure}
\begin{center}
\includegraphics[width=0.5\textwidth]{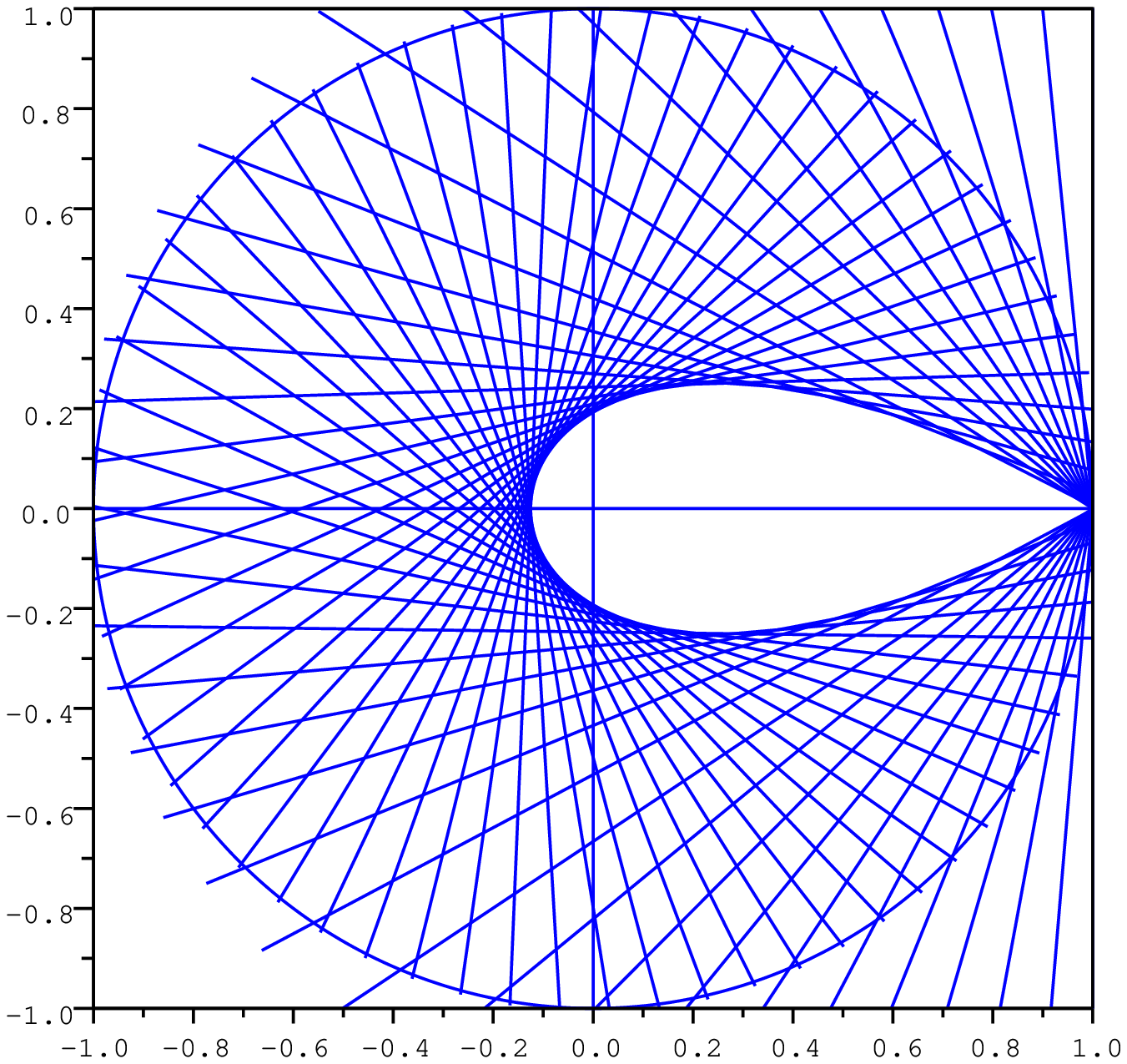}
\end{center}
\caption{\label{FigConElambda} $\Con{\lambda \in [0,1]}{\Ellipse{\lambda}}$}
\end{figure}

\begin{lemme}
  Let $P_1$ and $P_2$ be two orthogonal projections.
  Then $\adherence{\NumRan{P_2P_1}}$ is a subset of the rectangle whose sides are $x=-\frac{1}{8}$, $x=1$, $y=\frac{1}{4}$ and $y=-\frac{1}{4}$.
\end{lemme}

\begin{proof}
  Using Corollary \ref{CoroWP2P1Subset} and the parametric equation of the boundary of $\Ellipse{\lambda}$ (see Remark \ref{RqEqEllipse}), we can prove that for all $t \in \R$ and for all $\lambda \in [0,1]$, we have $ -\frac{1}{8} \le x_\lambda(t) \le 1 $ and $ -\frac{1}{4} \le y_\lambda(t) \le \frac{1}{4} $.
\end{proof}



\begin{proof}[Proof of Proposition \ref{LemWP2P1dsSecteur}]
  Suppose that we have found $\theta_\lambda$ such that 
  $
  \Ellipse{\lambda} \subset  \lbrace z \in \C, \abs{\arg(1-z)} \le \theta_\lambda \rbrace
  $ for every $\lambda$.
  Taking $\theta = \sup \{\theta_\lambda :  \lambda \in \Spec{P_2P_1}\} ,$
  we will have that
  $$  
  \NumRan{P_2P_1} \subset \Con{\lambda \in \Spec{P_2P_1}}{\Ellipse{\lambda}} \subset  \lbrace z \in \C, \abs{\arg(1-z)} \le \theta \rbrace
  .$$
  
  First we note that $\Ellipse{0}= \lbrace 0 \rbrace$ and $\Ellipse{1}=[0,1]$. So we have $\theta_0 = \theta_1 = 0$. 
  For $\lambda \in ]0,1[$, we denote $(x_\lambda(t), y_\lambda(t))$ the parametrization of the boundary of $\Ellipse{\lambda}$ given in Remark \ref{RqEqEllipse}. We denote $\theta_\lambda(t)$ the angle between the line connecting the points 0 and 1, and the one connecting points 1 and $(x_\lambda(t), y_\lambda(t))$. We have that $\theta_\lambda = \sup_{t \in \R}\abs{\theta_\lambda(t)}$, and
  \[
    \tan(\theta_\lambda(t))
    = \frac{y_\lambda(t)}{1-x_\lambda(t)} 
    = \frac{\sqrt{\lambda(1-\lambda)}\sin(t)}{2-\lambda-\sqrt{\lambda}\cos(t)}.
  \]  
  By differentiating $\tan(\theta_\lambda(t))$, we can see that $t_0$ is a critical point if $\cos(t_0)=\frac{\sqrt{\lambda}}{2-\lambda}$. So we have that
  $$
    \tan(\theta_\lambda)
    = \frac{\sqrt{\lambda(1-\lambda)}\sqrt{1-\frac{\lambda}{(2-\lambda)^2}}}{2-\lambda-\sqrt{\lambda}\frac{\sqrt{\lambda}}{(2-\lambda)}} 
    = \frac{\sqrt{\lambda(1-\lambda)}\sqrt{(2-\lambda)^2-\lambda}}{(2-\lambda)^2-\lambda} 
    = 
     \frac{\sqrt{\lambda}}{\sqrt{4-\lambda}}.
  $$
  As $\theta = \sup_{\lambda \in \Spec{P_2P_1}} \theta_\lambda$, we get that $\tan(\theta) = \sup_{\lambda \in \Spec{P_2P_1}\setminus \lbrace 1 \rbrace} \frac{\sqrt{\lambda}}{\sqrt{4-\lambda}}$. Then we conclude 
  using Lemma \ref{LCosSp}.
\end{proof}

\begin{remarque}
  We obtain as a consequence the result that the numerical range of a product of two orthogonal projections is included in a sector with vertex $1$ and angle $\pi/6$ (\cite{Crouzeix_2008}). Also, the result of 
  Proposition \ref{LemWP2P1dsSecteur} is sharp, in the sense that if $\theta < \arctan(\sqrt{\frac{\cos^2(M_1,M_2)}{4-\cos^2(M_1,M_2)}}) $, then $\NumRan{P_2P_1}$ is not included in $\lbrace z \in \C, \abs{\arg(1-z)} \le \theta \rbrace $.
\end{remarque}


\subsection{Some examples}

  Let $P_1,P_2$ be two orthogonal projections. The spectrum $\Spec{P_2P_1}$ is always a compact subset of $ [0,1]$. In this section, we study the following inverse spectral problem : let $K$ be a compact subset of $[0,1]$; when two orthogonal projections $P_1$ and $P_2$ exist such that $ \Spec{P_2P_1}=K $? We will show that the answer is positive if and only if $0 \in K$ or $K=\lbrace 1 \rbrace$.
  
  We start with the case $K=\lbrace 1 \rbrace$.
  
\begin{proposition}
  Let $M_1$ and $M_2$ be two subspaces of $H$.
  If $0$ does not belong to $\Spec{P_{M_2}P_{M_1}}$, then we have that $M_1=M_2=H$, $P_{M_1} = P_{M_2}= I$ and $\Spec{P_{M_2}P_{M_1}} = \lbrace 1 \rbrace$.
\end{proposition}

\begin{proof}
  We decompose $H$ as in (\ref{EqDecompoPosGen}):
  \[
  H= (M_1 \cap M_2) \oplus (M_1 \cap M_2^\perp) \oplus (M_1^\perp \cap M_2) \oplus (M_1^\perp \cap M_2^\perp) \oplus \tilde{H}.
  \]
  Then 
  $
  P_{M_2}P_{M_1} = I \oplus 0 \oplus 0 \oplus 0 \oplus P_2P_1
  $.  
  As $0$ does not belong to $\Spec{P_2P_1}$, we obtain
  $
  M_1 \cap M_2^\perp = M_1^\perp \cap M_2 = M_1^\perp \cap M_2^\perp = \tilde{H} =\lbrace 0 \rbrace
  $
  (otherwise $P_{M_2}P_{M_1}$ would have a non trivial kernel).
  So we have $H=M_1 \cap M_2$ and  $M_1=M_2=H$. Therefore $P_{M_1} = P_{M_2}= I$ and $\Spec{P_{M_2}P_{M_1}} = \Spec{I} = \lbrace 1 \rbrace$. 
\end{proof}

Now, we suppose that $0 \in K$.

\begin{theoreme}
  \label{ThKP2P1}
  Let $H$ be a separable Hilbert space. Let $K$ be a compact subset of $[0,1]$ such that $0 \in K$. Then there exist two orthogonal projections $P_1, P_2$ on $H$ such that $\Spec{P_2P_1}=K$. Moreover, $P_1P_2P_1$ is diagonalisable.
\end{theoreme}
  
\begin{proof}
  As $K$ is a compact subset of $[0,1]$, there exists a sequence  $(\lambda_n)$ in $K$ such that $\adherence{\lbrace \lambda_n, n \in \N \rbrace} = K$. For all $ n \in \N$, there exists a unique $\theta_n \in [0,\frac{\pi}{2}]$ such that $ \lambda_n = \cos(\theta_n)^2$.  
  Let $(e_n)_{n \in \N}$ be an orthonormal basis of $H$. We denote $h_n = e_{2n}$, $\tilde{f}_n= e_{2n+1}$ and $\tilde{w}_n = \cos(\theta_n) e_{2n} + \sin(\theta_n) e_{2n+1}$. Let $N_1=\Span{h_n, n \in \N } $ and $N_2=\Span{\tilde{w}_n, n \in \N} $ (see Figure \ref{FigDiag}).
  Then we have that $ P_1 h_n = h_n$, $P_1 \tilde{f}_n = 0$ and $  P_2 h_n = \cos(\theta_n)^2 h_n  +  \cos(\theta_n)\sin(\theta_n) \tilde{f}_n $, $  P_2 \tilde{f}_n = \cos(\theta_n)\sin(\theta_n) h_n  +  \sin(\theta_n)^2 \tilde{f}_n $. Hence $  P_2P_1 h_n = \cos(\theta_n)^2 h_n  +  \cos(\theta_n)\sin(\theta_n) \tilde{f}_n $ and $P_2P_1 \tilde{f}_n = 0 $. Thus we get
\begin{align*}
  P_2P_1 
    &= \bigoplus_{n \in \N} P_2P_1\mid_{\Span{h_n, \tilde{f}_n}} \\
    &= \bigoplus_{n \in \N}  \left( \begin{array}{cc}
\cos(\theta_n)^2 & 0 \\ 
\cos(\theta_n)\sin(\theta_n) & 0
\end{array} \right).
\end{align*}
Also,
$
  \Spec{P_2P_1} 
    = \adherence{\lbrace \cos(\theta_n)^2, n \in \N\rbrace \cup \lbrace 0 \rbrace } 
    = \adherence{\lbrace \lambda_n, n \in \N\rbrace \cup \lbrace 0 \rbrace } 
    = K
$.
\end{proof}

\begin{remarque}
 We have proved in the previous section that $ \adherence{\NumRan{P_2P_1}} \subset \adherence{\Con{}{\cup_{\lambda \in [0,1]}\Ellipse{\lambda}}}$. There are examples where this inclusion is an equality. According to Theorem \ref{CoroWP2P1}, we just need two projections that satisfy $\Spec{P_2P_1}=[0,1]$. The projections of Example \ref{RqExNonDiag} satisfy this condition, but $P_1P_2P_1$ is not diagonalisable. With Theorem \ref{ThKP2P1}, we can also construct an example such that $P_1P_2P_1$ is diagonalisable and $\Spec{P_2P_1}=[0,1]$.
\end{remarque}

\begin{remarque}
 As we now know all the possible shapes of $\Spec{P_2P_1} $, Theorem \ref{CoroWP2P1} gives all the possible shapes of $\adherence{\NumRan{P_2P_1}} $.
\end{remarque}

\begin{remarque}
  Using the parametrization of the boundary of $\Ellipse{\lambda}$ (see Remark \ref{RqEqEllipse}), we can prove that for all $\lambda \in [0,\frac{1}{4}]$, $\Ellipse{\lambda} \subset \Ellipse{\frac{1}{4}}$. Let $K_1=[0,\frac{1}{4}]$ and $K_2 = \lbrace 0,\frac{1}{4} \rbrace $. It follows from Theorem \ref{ThKP2P1} that there exist orthogonal projections $P_1,P_2,Q_1,Q_2$ such that $\Spec{P_2P_1}=K_1$ and $\Spec{Q_2Q_1}=K_2$. Moreover, we have that:
  \[
    \adherence{\NumRan{P_2P_1}}
      = \adherence{\Con{}{\cup_{\lambda \in [0,\frac{1}{4}]}\Ellipse{\lambda}}} 
      = \Ellipse{\frac{1}{4} } 
      = \adherence{\Con{}{\Ellipse{0} \cup \Ellipse{\frac{1}{4}}}} 
      = \adherence{\NumRan{Q_2Q_1}}.
  \]
  This shows that the points of the spectrum of $P_2P_1$ which are less that $\frac{1}{4}$ are not uniquely determined by the numerical range. We will see in the next section that the situation is different for spectral values greater than $\frac{1}{4}$.
\end{remarque}

\section{The spectrum of $P_2P_1$ in terms of the numerical range}

\subsection{The relationship between the spectral and numerical radii}


In this section, we will prove proposition \ref{PropLinkRaySpecNumRan}, and compare this result with an inequality from \cite{Kittaneh_2003}.

\begin{proof}[Proof of Proposition \ref{PropLinkRaySpecNumRan}]
  If $M_1=M_2=H$, this is true. Now we suppose that $M_1\ne H$ or $M_2 \ne H$.   By combining the definition of the numerical radius with the Theorem \ref{CoroWP2P1}, we obtain:
\[
  \NumRay{P_2P_1}
    = \sup_{w \in \adherence{\NumRan{P_2P_1}}} \abs{w} 
    = \sup_{w \in \Ellipse{\lambda}, \lambda \in \Spec{P_2P_1}} \abs{w} .
\]

  First, we compute $\sup_{w \in \Ellipse{\lambda}} \abs{w}$ for a fixed $\lambda$. 
  We denote by $(x_\lambda(t), y_\lambda(t))$ the parametrization of the boundary of $\Ellipse{\lambda}$ given in Remark \ref{RqEqEllipse}. We have 
  $
  \sup_{w \in \Ellipse{\lambda}} \abs{w} = \sup_{t \in \R} \sqrt{x_\lambda(t)^2 + y_\lambda(t)^2}
  $
  and
$
  x_\lambda(t)^2 + y_\lambda(t)^2 = \frac{1}{4} ( \lambda^2 \cos(t)^2 + 2\lambda\sqrt{\lambda}\cos(t) + \lambda ) 
$.
Therefore
\[
  \sup_{w \in \Ellipse{\lambda}} \abs{w}
    = \sqrt{ \frac{1}{4} \left( \lambda^2 + 2\lambda \sqrt{\lambda} + \lambda \right) } 
    = \frac{1}{2}(\lambda + \sqrt{\lambda} ).
\]
Finally, 
\[
  \NumRay{P_2P_1}
    = \sup_{w \in \Ellipse{\lambda}, \lambda \in \Spec{P_2P_1}} \abs{w} \\
    = \sup_{\lambda \in \Spec{P_2P_1}} \frac{1}{2}(\lambda + \sqrt{\lambda} ) \\
    =  \frac{1}{2}(\RaySpec{P_2P_1} + \sqrt{\RaySpec{P_2P_1}} ).
\]
\end{proof}

\begin{remarque}
  In \cite{Kittaneh_2003}, Kittaneh proved that for any operator $T$, we have the following inequality:
  \begin{equation}
    \label{EqKittaneh}
    \NumRay{T} \le \frac{1}{2} \left( \norme{T} + \norme{T^2}^{\frac{1}{2}} \right).
  \end{equation}
  Let us compare Proposition \ref{PropLinkRaySpecNumRan} with Kittaneh's inequality when $T=P_2P_1$. If $M_1 \cap M_2 \ne \lbrace 0 \rbrace$, then 1 is eigenvalue of $P_2P_1$. So $\norme{P_2P_1}= \norme{(P_2P_1)^2} = 1$, $\RaySpec{P_2P_1}=1$ and $\NumRay{P_2P_1}= 1 $. Thus $\NumRay{P_2P_1} = \frac{1}{2} (\sqrt{\RaySpec{P_2P_1}} + \RaySpec{P_2P_1}) = \frac{1}{2} (\norme{P_2P_1} + \norme{(P_2P_1)^2}^{\frac{1}{2}})$ and in this case, (\ref{EqKittaneh}) is an equality.
  
  If $M_1 \cap M_2 = \lbrace 0 \rbrace$, then according to \cite{Kaylar_Weinert_1988, Deustch_2001} we 
  have $\norme{(P_2P_1)^n}= \cos(M_1,M_2)^{2n-1} $ and $\norme{P_1P_2P_1}=  \cos(M_1,M_2)^2 = \RaySpec{P_1P_2P_1} = \RaySpec{P_2P_1} $. So we have
$
  \NumRay{P_2P_1}
    = \frac{1}{2} (\sqrt{\RaySpec{P_2P_1}} + \RaySpec{P_2P_1}) 
    = \frac{1}{2} (\cos(M_1,M_2) + \cos(M_1,M_2)^2)
$
  and also
$
   \frac{1}{2} (\norme{P_2P_1} + \norme{(P_2P_1)^2}^{\frac{1}{2}})
   =  \frac{1}{2} (\cos(M_1,M_2) + \cos(M_1,M_2)^{\frac{3}{2}})
$.
  If $\cos(M_1,M_2)<1 $, then 
  $
  \NumRay{P_2P_1} < \frac{1}{2} \left(\norme{P_2P_1} + \norme{(P_2P_1)^2}^{\frac{1}{2}}\right)
  $.
So in this case, (\ref{EqKittaneh}) is a strict inequality.
\end{remarque}

\subsection{How to find $\Spec{P_2P_1}$ from $\adherence{\NumRan{P_2P_1}}$ (and $\adherence{\NumRan{P_2(I-P_1)}}$)}

  Contrarily to Sections 2.1 and 2.2, where we have described  $\adherence{\NumRan{P_2P_1}}$ in terms of $\Spec{P_2P_1}$, the aim of this section is to obtain information about the spectrum of $P_2P_1$ from its numerical range. We give an informal idea about how we do this. Denote $g_\alpha(\lambda) = \frac{1}{2}(\cos(\alpha)\lambda + \sqrt{\lambda(1-\sin(\alpha)^2\lambda)})$, then we have $\FctionSupp{\NumRan{P_2P_1}}{\alpha} = \sup_{\lambda \in \Spec{P_2P_1}} g_\alpha(\lambda)$. We will use the support function as a tool to identify if the ellipse $\Ellipse{\lambda}$ is in the numerical range. If this is the case, then $\lambda$ will be in the spectrum.
Denote by $ \mathscr{S}$ the closure of $\Con{\lambda \in [0,1]}{\Ellipse{\lambda}} $. By Corollary \ref{CoroWP2P1Subset}, we have $\NumRan{P_2P_1} \subset \mathscr{S} $, so $ \sup_{\lambda \in \Spec{P_2P_1}} g_\alpha(\lambda) = \FctionSupp{\NumRan{P_2P_1}}{\alpha} \le \FctionSupp{\mathscr{S}}{\alpha} = \sup_{\lambda \in [0,1]} g_\alpha(\lambda)$. Using the continuity of the function $g_\alpha(\cdot)$ and the compacity of $\Spec{P_2P_1}$, we get the existence of a point $\lambda_0 \in \Spec{P_2P_1}$ such that $\FctionSupp{\NumRan{P_2P_1}}{\alpha} = g_\alpha(\lambda_0)$. With this information we are able to find an explicit formula for $ \FctionSupp{\mathscr{S}}{\alpha}$. Moreover, we will see that the equality $ \FctionSupp{\NumRan{P_2P_1}}{\alpha} = \FctionSupp{\mathscr{S}}{\alpha} $ is equivalent to the presence of a unique point $\lambda_0$ (depending only on $\alpha$) in the spectrum of $P_2P_1$. 

%
%
%

  We begin by giving a necessary and sufficient condition such that $\lambda$ is a critical point of $ g_\alpha(\lambda) $.
\begin{lemme}
  \label{LemCNSptcritiq}
  Let $\lambda_0 \in ]0,1[$ and $\alpha \in ]0,\pi[ $. Then $\lambda_0$ is a critical point for $ g_\alpha $ if and only if we have:
  \[
  \alpha = 2 \arcsin(\sqrt{1-\lambda_0\sin(\alpha)^2}).
  \]
\end{lemme} 

\begin{proof}
  We have
  $
  g'_\alpha(\lambda) = \frac{1}{2} ( \cos(\alpha) + \frac{1-2\lambda \sin(\alpha)^2}{2\sqrt{\lambda(1-\sin(\alpha)^2\lambda)}} )
  $.
Thus $g'_\alpha(\lambda) =0$ if and only if
  $
  \sqrt{\lambda}\cos(\alpha) = \frac{1}{2\sqrt{1-\sin(\alpha)^2\lambda}} - \sqrt{1-\lambda\sin(\alpha)^2}
  $. 
Denoting $X = \sqrt{1-\sin(\alpha)^2\lambda}$, we have that $g'_\alpha(\lambda) =0$ if and only if
  $
  \sqrt{\frac{1-X^2}{\sin(\alpha)^2}}\cos(\alpha) = \frac{1}{2X} - X 
  $,
or, equivalently, if and only if 
  $
  \cot(\alpha) = \frac{1-2X^2}{2X\sqrt{1-X^2}}
  $.
We denote $X=\sin(\gamma)$ and get that
$
  \cot(\alpha) 
    = \frac{1-2\sin(\gamma)^2}{2\sin(\gamma)\sqrt{1-\sin(\gamma)^2}} 
    = \cot(2\gamma)
$.
  As $\lambda \in [0,1]$, we have that $X \in [\abs{\cos(\alpha)},1]$, and $\gamma \in [\arcsin(\abs{\cos(\alpha)}),\frac{\pi}{2} ] \subset [0,\frac{\pi}{2}]$. So $2\gamma \in [0,\pi]$.
  Therefore  $g'_\alpha(\lambda) =0$ if and only if $\alpha = 2\gamma$, if and only if $\alpha = 2 \arcsin(\sqrt{1-\lambda_0\sin(\alpha)^2}) $.
\end{proof}  
  
  The next corollary says that the support functions of $\NumRan{P_2P_1}$ for $\alpha \in ]0,\frac{\pi}{3}]$ do not give us useful information about $\Spec{P_2P_1}$.
  
\begin{corollaire}
  \label{CoroPasPtCritiq}
  If $\alpha \in [0,\frac{\pi}{3}]$, and $\lambda_0 \in ]0,1[$, then $\lambda_0$ is not a critical point of $g_\alpha$.
\end{corollaire}

\begin{proof}
  We just need to check that Lemma \ref{LemCNSptcritiq} fails in this case.
  If $\lambda_0 \in ]0,1[$, then $2\arcsin(\sqrt{1-\lambda_0\sin(\alpha)^2}) \in ]\arcsin(\abs{\cos(\alpha)}) ,\pi[$. If $\alpha$ satisfies the condition of Lemma \ref{LemCNSptcritiq}, then $\alpha \in ]\arcsin(\abs{\cos(\alpha)}) ,\pi[$.
  We want to know when we have $\alpha = 2\arcsin(\abs{\cos(\alpha)}) $.   
  If $\alpha = 2\arcsin(\abs{\cos(\alpha)}) $, using some trigonometric formulas, we get that $\sin(\alpha)=2 \abs{\cos(\alpha)} \sin(\alpha) $.
  So $\abs{\cos(\alpha)} = \frac{1}{2} $. If $\alpha = \frac{\pi}{3}$, then $2\arcsin(\abs{\cos(\alpha)}) = \frac{\pi}{3}= \alpha$. If $\alpha = \frac{2\pi}{3}$, then $2\arcsin(\abs{\cos(\alpha)}) = \frac{\pi}{3} \neq \alpha$. In other words, $\alpha = 2\arcsin(\abs{\cos(\alpha)}) $ if and only if $\alpha = \frac{\pi}{3}$. Moreover, if $\alpha \in [0,\frac{\pi}{3}]$, then we have $\alpha < 2\arcsin(\abs{\cos(\alpha)}) $, so $g_\alpha$ has no critical point on $]0,1[$.
\end{proof}

The following proposition says that $\FctionSupp{\NumRan{P_2P_1}}{\alpha}$ can give information on $\Spec{P_2P_1}$ if $\alpha \in [\frac{\pi}{3},\pi]$.

\begin{proposition}
  \label{PropPtCritiq}
  If $\alpha \in [\frac{\pi}{3},\pi]$, then the only critical point of $g_\alpha$ is $\lambda_\alpha = \frac{1+\cos(\alpha)}{2\sin(\alpha)^2}$.
\end{proposition}
 
\begin{proof}
  From Lemma \ref{LemCNSptcritiq}, we know that $\lambda$ is a critical point of $g_\alpha$ if and only if $\alpha = 2 \arcsin(\sqrt{1-\lambda\sin(\alpha)^2})$. Compose with sinus on each side of the equality and use some trigonometric formulas to get that $\sin(\alpha) =2 \sqrt{1-\lambda\sin(\alpha)^2}\sqrt{\lambda}\sin(\alpha) $.  
  Dividing each side by $\sin(\alpha)$ and raising to the square, we get that $4\lambda^2\sin(\alpha)^2 - 4\lambda + 1 = 0 $. Therefore if $\lambda$ is a critical point of $g_\alpha$, then $\lambda = \frac{1+\cos(\alpha)}{2\sin(\alpha)^2}$ or $\lambda = \frac{1-\cos(\alpha)}{2\sin(\alpha)^2}$.
  If  $\lambda = \frac{1-\cos(\alpha)}{2\sin(\alpha)^2}$, then
\begin{align*}
  2 \arcsin(\sqrt{1-\lambda\sin(\alpha)^2})
    &= 2 \arcsin(\sqrt{\frac{1}{2}(1 + \cos(\alpha))}) \\
    &= 2 \arcsin(\cos(\frac{\alpha}{2})) \\
    &\ne \alpha . 
\end{align*}
  Lemma \ref{LemCNSptcritiq} says that $\lambda$ is not a critical point of $g_\alpha$.  
  If $\lambda = \frac{1+\cos(\alpha)}{2\sin(\alpha)^2}$, then
\begin{align*}
  2 \arcsin(\sqrt{1-\lambda\sin(\alpha)^2})
    &= 2 \arcsin(\sqrt{\frac{1}{2}(1 - \cos(\alpha))}) \\
    &= 2 \arcsin(\sin(\frac{\alpha}{2})) \\
    &= \alpha. 
\end{align*}
  According to Lemma \ref{LemCNSptcritiq}, $\lambda$ is a critical point of $g_\alpha$.   
\end{proof} 

\begin{remarque}
  The condition  $\alpha \in [\frac{\pi}{3},\pi]$ ensures that $\lambda_\alpha \in [0,1]$. We remark that
\[
  \lambda_\alpha
    = \frac{1+\cos(\alpha)}{2\sin(\alpha)^2}
    = \frac{1}{2(1-\cos(\alpha))}. 
\] 
  If we have $ \frac{\pi}{3} \le \alpha \le \pi $, then $\frac{1}{4} \le \frac{1}{2(1 - \cos(\alpha))} \le 1$.
  So $\lambda_\alpha \in [\frac{1}{4},1]$.
\end{remarque}

  We give now an explicit formula for $\FctionSupp{\mathscr{S}}{\alpha}$.

\begin{corollaire}
  The support function of $ \mathscr{S} = \adherence{\Con{}{\cup_{\lambda \in [0,1]}\Ellipse{\lambda}}} $ is given by the following formula:
  $$
  \FctionSupp{\mathscr{S}}{\alpha}
    =\left\lbrace \begin{matrix}
      \cos(\alpha) & \textrm{if } \alpha \in [0,\frac{\pi}{3}] \\
      \frac{1}{4(1-\cos(\alpha))} & \textrm{if } \alpha \in [\frac{\pi}{3},\pi] \\
    \end{matrix} \right. .  
  $$
\end{corollaire}

\begin{proof}
  We know that $\FctionSupp{\mathscr{S}}{\alpha} = \max_{\lambda \in [0,1]} g_\alpha(\lambda)$. We proved previously that if $ \alpha \in [0,\frac{\pi}{3}]$, then
  $
  \FctionSupp{\mathscr{S}}{\alpha} = \max \lbrace g_\alpha(0), g_\alpha(1) \rbrace
  $
  and if $ \alpha \in [\frac{\pi}{3},\pi]$ then
  $
  \FctionSupp{\mathscr{S}}{\alpha} = \max \lbrace g_\alpha(0),g_\alpha(\lambda_\alpha), g_\alpha(1) \rbrace
  $,
  with $\lambda_\alpha = \frac{1}{2(1-\cos(\alpha))}$. We have that $g_\alpha(0) = 0$ and $ g_\alpha(1)=\cos(\alpha)$ and also $g_\alpha(\lambda_\alpha) = \frac{1}{4(1-\cos(\alpha))}$.
  Now  it remains to show that for any $ \alpha \in [\frac{\pi}{3},\pi]$, we have $ g_\alpha(\lambda_\alpha) \ge g_\alpha(1)$. As
$
  \frac{1}{4(1-\cos(\alpha))} - \cos(\alpha)
    = \frac{1-4\cos(\alpha) + 4\cos(\alpha)^2 }{4(1-\cos(\alpha))} 
    = \frac{(1-2\cos(\alpha))^2}{4(1-\cos(\alpha))} 
$,
  and the last term is always positive, we get the announced result.
\end{proof}

Now we have enough material to prove Theorem \ref{ThLienSpecWP2P1}.

\begin{proof}[Proof of Theorem \ref{ThLienSpecWP2P1}]
  Let $\alpha \in [\frac{\pi}{3},\pi]$. We know that $\FctionSupp{\NumRan{P_2P_1}}{\alpha} = \sup_{\lambda \in \Spec{P_2P_1}} g_\alpha(\lambda) $. As $\Spec{P_2P_1}$ is a compact set and $g_\alpha$ is a continuous function, there exists a $\lambda_0 \in \Spec{P_2P_1}$ such that:
  $
  \FctionSupp{\NumRan{P_2P_1}}{\alpha} = \max_{\lambda \in \Spec{P_2P_1}} g_\alpha(\lambda) 
  = g_\alpha(\lambda_0)
  $.
  According to Proposition \ref{PropPtCritiq}, we have $g_\alpha(\lambda_0) = \frac{1}{4(1-\cos(\alpha))}$ if and only if $\lambda_0= \lambda_\alpha = \frac{1}{2(1-\cos(\alpha))} $. 
   
  $"1 \Rightarrow 2"$:  If $ \FctionSupp{\NumRan{P_2P_1}}{\alpha} = \frac{1}{4(1-\cos(\alpha))} = g_\alpha(\lambda_0)$, then we have $\lambda_0= \lambda_\alpha = \frac{1}{2(1-\cos(\alpha))} $. As $\lambda_0 \in \Spec{P_2P_1}$, we get that $\lambda_\alpha \in \Spec{P_2P_1} $. 
   
  $"2 \Rightarrow 1"$: If $\lambda_\alpha \in \Spec{P_2P_1} $, then we have that:
  \[
  g_\alpha(\lambda_\alpha) 
  \le \max_{\lambda \in \Spec{P_2P_1}} g_\alpha(\lambda)
  \le \max_{\lambda \in [0,1]} g_\alpha(\lambda)
  = g_\alpha(\lambda_\alpha).
  \]  
  Therefore
  \[
    \FctionSupp{\NumRan{P_2P_1}}{\alpha}
      = \max_{\lambda \in \Spec{P_2P_1}} g_\alpha(\lambda) 
      = g_\alpha(\lambda_\alpha) 
      = \frac{1}{4(1-\cos(\alpha))} .
  \]
\end{proof}

  Given $\alpha$, Theorem \ref{ThLienSpecWP2P1} tells us whether $\lambda_\alpha$ is in the spectrum or not by looking at the support function of $\NumRan{P_2P_1}$ in the direction $\alpha $. Given $\lambda$, the next corollary tell us in which direction $\alpha_\lambda$ we have to look to know whether $\lambda$ is in $\Spec{P_2P_1}$ or not. 

\begin{corollaire}
  Let $\lambda \in [\frac{1}{4},1]$. We denote $\alpha_\lambda = \arccos(1-\frac{1}{2\lambda})$. The following assertions are equivalent:
  \begin{enumerate}
    \item $\FctionSupp{\NumRan{P_2P_1}}{\alpha_\lambda} = \frac{1}{4(1-\cos(\alpha_\lambda))} $ ;
    \item $\lambda \in \Spec{P_2P_1} $.
  \end{enumerate}
\end{corollaire}

\begin{proof}
  We denote $f: [\frac{\pi}{3},\pi]\longrightarrow [\frac{1}{4},1]$ the function given by $f(\alpha)= \frac{1}{2(1-\cos(\alpha))}$. The equivalence follows from Theorem \ref{ThLienSpecWP2P1}, and the facts that $f$ is bijective with inverse function given by  $\lambda \mapsto \arccos(1-\frac{1}{2\lambda})$.
\end{proof}

  The next proposition is a "trick" to deduce most of the spectrum of $P_2P_1$ from $ \Spec{P_2(I-P_1)}$. As $P_2(I-P_1) $ is again a product of two orthogonal projections,  all the results of this paper apply also to this operator.

\begin{proposition}
  \label{PropIP1versP1}
  Let $\lambda \ne 0$. If $\lambda \in \Spec{P_2(I-P_1)}$, then $1-\lambda \in \Spec{P_2P_1} $.
\end{proposition}

\begin{proof}
  We decompose $H$ as in (\ref{EqDecompoPosGen}). Therefore we have
  \[
  H= (M_1 \cap M_2) \oplus (M_1 \cap M_2^\perp) \oplus (M_1^\perp \cap M_2) \oplus (M_1^\perp \cap M_2^\perp) \oplus \tilde{H}
  \]
  and
\begin{align*}
  P_{M_1} &\UnitEq I \oplus I \oplus 0 \oplus 0 \oplus \left(\begin{array}{cc}
I & 0 \\ 
0 & 0
\end{array} \right) \\
  P_{M_2} &\UnitEq I \oplus 0 \oplus I \oplus 0 \oplus \left(\begin{array}{cc}
C^2 & CS \\ 
CS & S^2
\end{array} \right) \\
  I-P_{M_1} &\UnitEq 0 \oplus 0 \oplus I \oplus I \oplus \left(\begin{array}{cc}
0 & 0 \\ 
0 & I
\end{array} \right) \\
  P_{M_2}P_{M_1} &\UnitEq I \oplus 0 \oplus 0 \oplus 0 \oplus \left(\begin{array}{cc}
C^2 & 0 \\ 
CS & 0
\end{array} \right) \\
  P_{M_2}(I-P_{M_1}) &\UnitEq 0 \oplus 0 \oplus I \oplus 0 \oplus \left(\begin{array}{cc}
0 & CS \\ 
0 & S^2
\end{array} \right) 
\end{align*}
  We remind that $C^2 + S^2=I$, so we have that $\Spec{S^2}= 1- \Spec{C^2}$. Suppose that the subspaces $M_1^{(\perp)} \cap M_2^{(\perp)}$ and $ \tilde{H} $ are not equal to $\lbrace 0 \rbrace$. Then $\Spec{P_{M_2}P_{M_1}} = \cup\lbrace \lbrace 1 \rbrace,\lbrace 0 \rbrace, \Spec{C^2}\cup\lbrace  0\rbrace  \rbrace $. If $M_1 \cap M_2 = \lbrace 0 \rbrace $, then we have to remove $\lbrace 1 \rbrace $ of the former union to get $\Spec{P_{M_2}P_{M_1}} $. If $ M_1 \cap M_2^\perp = M_1^\perp \cap M_2 = M_1^\perp \cap M_2^\perp = \lbrace 0 \rbrace $, then we have to remove $\lbrace 0 \rbrace $ of the former union to get $\Spec{P_{M_2}P_{M_1}} $. If $ \tilde{H} = \lbrace 0 \rbrace $, then we have to remove $ \Spec{C^2}\cup\lbrace  0\rbrace $ of the former union to get $\Spec{P_{M_2}P_{M_1}} $. 
  
   In a similar way, $\Spec{P_{M_2}(I-P_{M_1})} = \cup\lbrace \lbrace 0 \rbrace,\lbrace 1 \rbrace, (1-\Spec{C^2})\cup\lbrace  0\rbrace  \rbrace $ depending on whether the corresponding subspaces are not reduced to $\lbrace 0 \rbrace $.
  
  Let $\lambda \neq 0$ be such that $\lambda \in \Spec{P_{M_2}(I-P_{M_1})}$. Suppose that $\lambda=1$ and $M_1^\perp \cap M_2 \neq \lbrace 0 \rbrace $. Then we get that $P_{M_2}P_{M_1}=0 $ on $M_1^\perp \cap M_2$. So $1-\lambda=0$ is an eigenvalue of $P_{M_2}P_{M_1}$.
  
  In the other cases, we get that $\lambda \in 1-\Spec{C^2}$ and $\tilde{H} \ne \lbrace 0 \rbrace$, hence $1-\lambda \in \Spec{C^2} \subset \Spec{P_{M_2}P_{M_1}}$.
\end{proof}

\begin{exemple}
  There exist orthogonal projections such that $1-\Spec{P_{M_2}(I-P_{M_1})} \neq \Spec{P_{M_2}P_{M_1}} $. We will exhibit an example in $H=\C^3$. Let $(e_1,e_2,e_3)$ be an orthonormal basis of $\C^3$. We set $M_1 = \Sspan{e_1} $ and $M_2 = \Sspan{e_2} $. Then we get that $M_1 \cap M_2 = \lbrace 0 \rbrace$, $ M_1 \cap M_2^\perp = \Sspan{e_2}$,  $ M_1^\perp \cap M_2 = \Sspan{e_1}$, $M_1^\perp \cap M_2^\perp = \Sspan{e_3}$ and $\tilde{H} = \lbrace 0 \rbrace$. So $P_{M_2}P_{M_1}=0 $, $P_{M_2}(I-P_{M_1}) = P_{M_2} $, $ \Spec{P_{M_2}P_{M_1}} = \lbrace 0 \rbrace$ and $ \Spec{P_{M_2}(I-P_{M_1})} = \lbrace 0,1 \rbrace $.  
\end{exemple}

\begin{remarque}
  \label{RqLienSpecWP2P1}
  Theorem \ref{ThLienSpecWP2P1} allows us to deduce $\Spec{P_2P_1}\cap [\frac{1}{4},1]$ from $\adherence{\NumRan{P_2P_1}}$. As $I-P_1$ is also an orthogonal projection, we can also deduce $\Spec{P_2(I-P_1)}\cap [\frac{1}{4},1]$ from $\adherence{\NumRan{P_2(I-P_1)}}$. Moreover, Proposition \ref{PropIP1versP1} allows us to deduce $\Spec{P_2P_1}\cap [0,\frac{3}{4}]$ from $\Spec{P_2(I-P_1)}\cap [\frac{1}{4},1]$. 
  
  In other words, we can deduce $\Spec{P_2P_1}$ from $\adherence{\NumRan{P_2P_1}}$ and $\adherence{\NumRan{P_2(I-P_1)}}$.
\end{remarque}

\begin{proposition}
  Let $P_1, P_2$ be two orthogonal projections. If $ \alpha \in [0, \frac{\pi}{2} ]$, then we have that
  \[
  \FctionSupp{\NumRan{P_2P_1}}{\alpha} = \RaySpec{\Reel{\exp(-\I \alpha)P_2P_1}} = \norme{\Reel{\exp(-\I \alpha)P_2P_1}} =\NumRay{\Reel{\exp(-\I \alpha)P_2P_1}}.
  \]
\end{proposition}

This proposition is significant because if we know $\RaySpec{\Reel{\exp(-\I \alpha)P_2P_1}}$ and $\RaySpec{\Reel{\exp(-\I \alpha)P_2(I-P_1)}}$ for every $ \alpha \in [\frac{\pi}{3},\frac{\pi}{2}]$ then, by using Theorem \ref{ThLienSpecWP2P1} and Proposition \ref{PropIP1versP1}, we can deduce $\Spec{P_2P_1}$.

\begin{proof}
  Notice that $\Reel{\exp(-\I \alpha)P_2P_1} $ is an hermitian operator, so $\RaySpec{\Reel{\exp(-\I \alpha)P_2P_1}} = \norme{\Reel{\exp(-\I \alpha)P_2P_1}}=\NumRay{\Reel{\exp(-\I \alpha)P_2P_1}}$ and the highest positive spectral value of $\Reel{\exp(-\I \alpha)P_2P_1} $ is the highest positive value in the numerical range. In other words, we just need to prove that for all $ \alpha \in [0, \frac{\pi}{2} ]$, the highest positive spectral value of $\Reel{\exp(-\I \alpha)P_2P_1} $ is greater than its lowest negative spectral value.
  
  According to the notation of the end of the proof of Lemma \ref{RqriTilde} (i.e. $\tilde{v_i}(\lambda,\alpha) = \frac{1}{2}(\cos(\alpha)\lambda \pm \sqrt{\lambda(1-\sin(\alpha)^2\lambda)})$ ), we have that
  $$
  \Reel{\exp(-\I \alpha)P_2P_1} 
  \UnitEq \left( \begin{array}{cc}
  \tilde{v_1}(C^2,\alpha) & 0 \\ 
  0 & \tilde{v_2}(C^2,\alpha)
  \end{array} \right).
  $$
  We also have that for all $ \lambda \in [0,1]$ and for all $\alpha \in [0,\pi]$, $ \tilde{v_1}(\lambda,\alpha) \ge 0$ and $\tilde{v_2}(\lambda,\alpha) \le 0$. Moreover $\lambda \in \Spec{C^2}$ if and only if $\tilde{v_1}(\lambda,\alpha)$ and $\tilde{v_2}(\lambda,\alpha) \in \Spec{ \Reel{\exp(-\I \alpha)P_2P_1} } $. Therefore
  $
  \abs{\tilde{v_1}(\lambda,\alpha)} - \abs{\tilde{v_2}(\lambda,\alpha)}
    = \tilde{v_1}(\lambda,\alpha) + \tilde{v_2}(\lambda,\alpha) 
    = \lambda \cos(\alpha)
  $.
  This last term is positive if $\alpha \in [0, \frac{\pi}{2}]$ and negative if $\alpha \in [\frac{\pi}{2},\pi] $. So $\alpha \in [0, \frac{\pi}{2}]$ implies that 
$
  \FctionSupp{\NumRan{P_2P_1}}{\alpha}
    = \RaySpec{\Reel{\exp(-\I \alpha)P_2P_1}} 
$.
\end{proof}

\section{Applications to the rate of convergence in the von Neumann-Halperin theorem and to the uncertainty principle}

\subsection{Applications to the method of alternating projections}
 Von Neumann proved (cf. \cite[Chapter 9]{Deustch_2001}) the following theorem:
\begin{theoreme}
  Let $ M_1, M_2 $ be two closed subspaces of $ H $. 
  Then for every $ x \in H $ we have that:
  \[
  \lim_{n \rightarrow \infty} \norme{(P_{M_2}P_{M_1})^n x - P_{M_1\cap M_2}x} = 0.
  \]
\end{theoreme} 

  If we set $ N_1 = M_1 \cap (M_1 \cap M_2)^\perp $ and $ N_2 = M_2 \cap (M_1 \cap M_2)^\perp $, we have that $ N_1 \cap N_2 = \lbrace 0 \rbrace $. In addition, we have
  \[
    (P_{M_2}P_{M_1})^n - P_{M_1\cap M_2} = (P_{N_2}P_{N_1})^n
  \]
  for every $ n \in \N $.
  Therefore, the study of the convergence of $ (P_{M_2} P_{M_1})^n $ to $ P_{M_1 \cap M_2} $ reduces to studying the convergence of $ (P_ {N_2} P_ {N_1}) ^ n $ to $ 0 $. 
  
  If one looks at the speed of convergence of $ (P_{N_2} P_{N_1})^n $ to 0, we have the dichotomy that either $(P_{N_2}P_{N_1})^n$ converges linearly to 0, or $(P_{N_2}P_{N_1})^n$ converges \emph{arbitrarily slowly} to 0.  We can characterize arbitrarily slow convergence in many ways;  see \cite{Bauschke_Deutsch_Hundal_2009, Catalin_Sophie_Muller_2010,Deustch_Hundal_2010_I,Deustch_Hundal_2010_II} and the references therein.
 

  
The novelty of the following characterization of arbitrarily slow convergence is in the use of the numerical range of $P_{N_2}P_{N_1}$ in items $6$ through $8$.

\begin{proposition}
  \label{PropEqCatalinSophieMuller}
 Let $ N_1, N_2 $ be two closed subspaces of $ H $ such that $ N_1 \cap N_2 = \lbrace 0 \rbrace $. The following assertions are equivalent:
  \begin{enumerate}
   \item $(P_{N_2}P_{N_1})^n$ converges arbitrarily slowly to 0 
    \item $\norme{P_{N_2}P_{N_1}}=1$ 
    \item $N_1^\perp + N_2^\perp $ is not closed 
    \item $ 1 \in \Spec{P_{N_2}P_{N_1}} $ 
    \item $\cos(N_1,N_2)=1  $
    \item[6.] $1 \in \adherence{\NumRan{P_{N_2}P_{N_1}}}$ 
    \item[7.] there exists a sequence $(\lambda_n)$ in $[0,1[$ such that $\lim \lambda_n = 1$ and for every $n \in \N$, $\Ellipse{\lambda_n} \subset \adherence{\NumRan{P_2P_1}}$
    \item[8.] there exists $\theta < \frac{\pi}{6} $ such that $\NumRan{P_{N_2}P_{N_1}} \subset \lbrace z \in \C, \abs{\arg(1-z)} \le \theta \rbrace $.
  \end{enumerate}
\end{proposition}

\begin{proof}
We refer to \cite{Bauschke_Deutsch_Hundal_2009, Catalin_Sophie_Muller_2010} (see also \cite[Chapter 9]{Deustch_2001}) for a proof of the equivalences of the first five assertions.

  $" 6 \Rightarrow 2"$. As $ 1 \in \adherence{\NumRan{P_{N_2}P_{N_1}}} $, we can find a sequence $ (x_n) $ such that $ \norme{x_n} = 1 $ and $ \lim_{n \rightarrow \infty} \ps{P_{N_2}P_{N_1}x_n}{x_n} = 1 $. Since we have that
  \begin{align*}
    \ps{P_{N_2}P_{N_1}x_n}{x_n}
      &\le \norme{P_{N_2}P_{N_1}x_n}\norme{x_n} \\
      &\le \norme{P_{N_2}P_{N_1}x_n} \\
      &\le \norme{P_{N_2}P_{N_1}} \\
      &\le 1,
  \end{align*}
  we have that $\norme{P_{N_2}P_{N_1}}=1$.
  
  $ "4 \Rightarrow 6" $. As $  1 \in \Spec{P_{N_2}P_{N_1}} $ and $\Spec{P_{N_2}P_{N_1}} \subset \adherence{\NumRan{P_{N_2}P_{N_1}}} $, we have that $ 1 \in \adherence{\NumRan{P_{N_2}P_{N_1}}} $.
  
  $ "7 \Rightarrow 6"$. This is clear as $ x_{\lambda_n}(0) = \frac{\sqrt{\lambda_n}}{2} + \frac{\lambda_n}{2} \in \Ellipse{\lambda_n} \subset \adherence{\NumRan{P_2P_1}} $.
  
  $"4 \Rightarrow 7" $. As $N_1 \cap N_2 = \lbrace 0 \rbrace $, 1 is not an eigenvalue of $P_{N_2}P_{N_1}$. So there exist $\lambda_n \in \Spec{P_{N_2}P_{N_1}}$ such that $\lim_n \lambda_n = 1$. The assertion $7$ follows from Theorem \ref{CoroWP2P1}.

  $"5 \Leftrightarrow 8" $. This is a consequence of Lemma \ref{LemWP2P1dsSecteur}.
\end{proof} 
  
\begin{remarque}
  In the spirit of \cite{Catalin_Sophie_Muller_2010}, we can extend $"1 \Leftrightarrow 6"$ to a finite number of projection, to obtain the following statement: If $P_{N_1}, \dots, P_{N_r} $ are orthogonal projections such that $\cap_{i=1}^r N_i =\lbrace 0 \rbrace$, then $(P_{N_r} \dots P_{N_1})^n$ converges arbitrarily slowly to 0 if and only if $1 \in \adherence{\NumRan{P_{N_r} \dots P_{N_1}}}$. The proof is similar.
\end{remarque}

\begin{remarque}
  The equivalences between items 5 through 8 still hold if we drop the assumption that $ N_1 \cap N_2 = \lbrace 0 \rbrace $.
\end{remarque}

  \subsection{Applications to annihilating pairs} 
  In this section we will give new characterizations of annihilating pairs. First we recall the context. We denote by $\F$ the Fourier transform on $\Ldeux(\R)$. Let $S$ and $\Sigma$ be two measurable subsets of $\R $. We denote by $M_g$ the operator of multiplication by $g \in \Linf(\R)$ (i.e.: $M_g(f) = gf $ for $f \in \Ldeux(\R)$). We denote by $\ind{S}$ the indicator function of the subset $S$. Set $P_S = M_{\ind{S}} $ and $P_\Sigma = \F^* M_{\ind{\Sigma}}\F $.

\begin{definition}
  We say that $(S,\Sigma)$ is an \emph{annihilating pair} if for every $f \in \Ldeux(\R) $ we have:
  $$
  P_S f = P_\Sigma f = f \Rightarrow f=0.
  $$
\end{definition}

\begin{definition}
  We say that $(S,\Sigma)$ is a \emph{strong annihilating pair} if there exists a constant $ c>0$ depending on $S,\Sigma$ such that for all $f \in \Ldeux(\R) $ we have:
  $$
  \norme{f}^2 \le c \left( \norme{(I-P_S)f}^2 + \norme{(I-P_\Sigma)f}^2 \right).
  $$
\end{definition}

  We want to recall some known facts (\cite{Havin_Joricke}, and \cite{Lenard_1971}) 
  about (strong) annihilating pairs.

\begin{proposition}
  \label{PropAPair}
  The following assertions are equivalents:
  \begin{itemize}
    \item[1.]  $(S,\Sigma)$ is an annihilating pair
    \item[2.]  $1+\I \notin \NumRan{P_S + \I P_\Sigma}$
    \item[3.]  $\Range{P_S} \cap \Range{P_\Sigma} = \lbrace 0 \rbrace$.
  \end{itemize}
\end{proposition}

\begin{proposition}
  \label{PropSAPair}
  The following assertions are equivalents:
  \begin{itemize}
    \item[a.]  $(S,\Sigma)$ is a strong annihilating pair
    \item[b.]  $1+\I \notin \adherence{\NumRan{P_S + \I P_\Sigma}}$
    \item[c.]  $\Range{P_S} \cap \Range{P_\Sigma} = \lbrace 0 \rbrace$ and $\cos(P_S,P_\Sigma) < 1 $
    \item[d.]  $\norme{P_SP_\Sigma} < 1$
    \item[e.]  $\RaySpec{P_SP_\Sigma} < 1$
    \item[f.]  $1 \notin \Spec{P_SP_\Sigma}$.
  \end{itemize}
\end{proposition}


The following proposition is a new characterization of annihilating pairs.

\begin{proposition}
  The following assertions are equivalent to the assertions of Proposition \ref{PropAPair}:
  \begin{itemize}
    \item[1.]  $(S,\Sigma)$ is an annihilating pair
    \item[4.]  $1 \notin \NumRan{P_SP_\Sigma} $.
  \end{itemize}
\end{proposition}

\begin{proof}
  We have that $1 \in \NumRan{P_SP_\Sigma} $ if and only if there exist $h \in H$ such that $\norme{h}=1$ ad $\ps{P_SP_\Sigma h}{h}=1 $. This is equivalent to the existence of some $h \in H$ such that $\norme{P_SP_\Sigma h}=\norme{h}=1 $.  This last assertion is equivalent to the negation of (3) in Proposition \ref{PropAPair}.
\end{proof}

\begin{proposition}
  The following assertions are equivalent to the assertions of Proposition \ref{PropSAPair}:
  \begin{itemize}
    \item[a.]  $(S,\Sigma)$ is a strong annihilating pair
    \item[g.]  $1 \notin \adherence{\NumRan{P_SP_\Sigma}} $
    \item[h.]  $\NumRay{P_SP_\Sigma} < 1$
    \item[i.]  for all $\alpha \in [0,\frac{\pi}{3}], \NumRay{\Reel{\exp(-\I \alpha) P_SP_\Sigma}}< \cos(\alpha) $
    \item[j.]  there exists $\alpha \in [0,\frac{\pi}{3}]$ such that $\NumRay{\Reel{\exp(-\I \alpha) P_SP_\Sigma}}< \cos(\alpha)$
    \item[k.]  there exists $\theta < \frac{\pi}{6} $ such that $\NumRan{P_SP_\Sigma} \subset \lbrace z \in \C, \abs{\arg(1-z)}\le \theta \rbrace \setminus \lbrace 1 \rbrace $.
  \end{itemize}
\end{proposition}

\begin{proof}
  $``f \Leftrightarrow g''$. By Theorem \ref{CoroWP2P1}, $1 \in  \adherence{\NumRan{P_SP_\Sigma}} $ if and only if $\Ellipse{1} \subset \adherence{\NumRan{P_SP_\Sigma}} $, if and only if $1 \in \Spec{P_SP_\Sigma}$.

  $``e \Leftrightarrow h''$. This is a direct consequence of Proposition \ref{PropLinkRaySpecNumRan}.

  $``f \Rightarrow i''$. This is a consequence of Corollary \ref{CoroPasPtCritiq}.

  $``i \Rightarrow j''$ This is trivial.

  $``j \Rightarrow f''$ This is a consequence of Corollary \ref{CoroPasPtCritiq}.

  $``c \Leftrightarrow k''$ This consequence of Lemma \ref{LemWP2P1dsSecteur}, and of the previous Proposition.
\end{proof}

\section*{Acknowledgment}

  I would like to thank Catalin Badea for several discussions and for his help to improve this paper, and Gustavo Corach for pointing out to me the content of Remark \ref{RqWIdempotent}. I would also like to thank the referee for the careful reading of the manuscript and helpful comments.

\bibliographystyle{alpha}
\bibliography{Biblio}

\begin{thebibliography}{BGM10}

\bibitem[BDH09]{Bauschke_Deutsch_Hundal_2009}
Heinz~H. Bauschke, Frank Deutsch, and Hein Hundal.
\newblock {Characterizing arbitrarily slow convergence in the method of
  alternating projections.}
\newblock {\em Int. Trans. Oper. Res. 16, no. 4, 413–425.}, 2009.

\bibitem[BGM]{Catalin_Sophie_Muller_2010}
Catalin Badea, Sophie Grivaux, and Vladimir M\"{u}ller.
\newblock {The rate of convergence in the method of alternating projections.}
\newblock {\em Algebra i Analiz 23 (2011), no. 3, 1--30; translation in St.
  Petersburg Math. J. 23 (2012), no. 3, 413–434}.

\bibitem[BGM10]{Catalin_Sophie_Muller_2010_bis}
Catalin Badea, Sophie Grivaux, and Vladimir M\"{u}ller.
\newblock {A generalization of the Friederichs angle and the method of
  alternating projections}.
\newblock {\em C. R. Math. Acad. Sci. Paris 348, no. 1-2, 53–56.}, 2010.

\bibitem[BL10]{Catalin_Lyubich_2010}
Catalin Badea and Yuri Lyubich.
\newblock {Geometric, spectral and asymptotic properties of averaged products
  of projections in Banach spaces}.
\newblock {\em Studia Math. 201, no. 1, 21–35.}, 2010.

\bibitem[BS10]{Bottcher_Spitkovsky_2010}
A.~Bottcher and I.M. Spitkovsky.
\newblock {A gentle guide to the basics of two projections theory}.
\newblock {\em Linear Algebra Appl. 432, no. 6, 1412–1459.}, 2010.

\bibitem[CM11]{Corach_Maestripieri_2011}
G.~Corach and A.~Maestripieri.
\newblock {Products of orthogonal projections and polar decomposition}.
\newblock {\em Linear Algebra Appl. 434, no. 6, 1594–1609.}, 2011.

\bibitem[Coh07]{Cohen_2007}
Guy Cohen.
\newblock {Iterates of a product of conditional expectation operators.}
\newblock {\em J. Funct. Anal. 242, no. 2, 658–668.}, 2007.

\bibitem[Cro07]{Crouzeix_2007}
Michel Crouzeix.
\newblock {Numerical Range and functional calculus in Hilbert space}.
\newblock {\em J. Funct. Anal. 244, no. 2, 668–690.}, 2007.

\bibitem[Cro08]{Crouzeix_2008}
Michel Crouzeix.
\newblock {A functional calculus based on the numerical range: applications}.
\newblock {\em Linear Multilinear Algebra 56, no. 1-2, 81–103.}, 2008.

\bibitem[DD99]{Delyon_1999}
Bernard Delyon and François Delyon.
\newblock {Generalization of von Neumann's spectral sets and integral
  representation of operators.}
\newblock {\em Bull. Soc. Math. France 127, no. 1, 25–41.}, 1999.

\bibitem[Deu01]{Deustch_2001}
Frank Deutsch.
\newblock {\em {Best Approximation in Inner Product Spaces}}.
\newblock Springer-Verlag, New York, 2001.

\bibitem[DH10a]{Deustch_Hundal_2010_I}
Frank Deutsch and Hein Hundal.
\newblock {Slow convergence of sequences of linear operators I, arbitrarily
  slow convergence}.
\newblock {\em J. Approx. Theory 162, no. 9, 1701–1716.}, 2010.

\bibitem[DH10b]{Deustch_Hundal_2010_II}
Frank Deutsch and Hein Hundal.
\newblock {Slow convergence of sequences of linear operators II, arbitrarily
  slow convergence}.
\newblock {\em J. Approx. Theory 162, no. 9, 1717–1738.}, 2010.

\bibitem[Gal04]{GalantaiBook}
A.~Galántai.
\newblock {\em {Projectors and projection methods.}}
\newblock Kluwer Academic Publishers, Boston, MA, 2004.

\bibitem[Gal08]{Galantai}
A.~Galántai.
\newblock {Subspaces, angles and pairs of orthogonal projections.}
\newblock {\em Linear Multilinear Algebra 56, no. 3, 227–260.}, 2008.

\bibitem[GR97]{Gustafson_Rao}
Karl~E. Gustafson and Duggirala~K.M. Rao.
\newblock {\em {Numerical Range}}.
\newblock Springer, 1997.

\bibitem[Hal69]{Halmos_1969}
Paul~R Halmos.
\newblock {Two Subspaces}.
\newblock {\em Trans. Amer. Math. Soc.,144, 381–389.}, 1969.

\bibitem[HJ94]{Havin_Joricke}
Victor Havin and Burglind Joricke.
\newblock {\em {The Uncertainty Principle in Harmonic Analysis}}.
\newblock Springer-Verlag, 1994.

\bibitem[Kit03]{Kittaneh_2003}
Fuad Kittaneh.
\newblock {A numerical radius inequality and an estimate for the numerical
  radius of the Froebenius companion}.
\newblock {\em Studia Math. 158, no. 1, 11–17.}, 2003.

\bibitem[KW88]{Kaylar_Weinert_1988}
S.~Kayalar and H.L. Weinert.
\newblock {Error bounds for the method of alternating projections}.
\newblock {\em Math. Control Signals Systems 1, no. 1, 43–59.}, 1988.

\bibitem[Len72]{Lenard_1971}
Andrew Lenard.
\newblock {The numerical range of a pair of projection}.
\newblock {\em J. Functional Analysis 10 (1972), 410–423.}, 1972.

\bibitem[Lum61]{Lumer_1961}
G.~Lumer.
\newblock {Semi inner product spaces}.
\newblock {\em Trans. Amer. Math. Soc. 100 1961 29–43.}, 1961.

\bibitem[Nee99]{Nees_1999}
Manuela Nees.
\newblock {Products of orthogonal projections as Carleman operators}.
\newblock {\em Integral Equations Operator Theory 35, no. 1, 85–92.}, 1999.

\bibitem[NN87]{Nelson_Neumann_1987}
Stuart Nelson and Michael Neumann.
\newblock {Generalisations of the projection method with applications to SOR
  theory for Hermitian positive semi definite linear system}.
\newblock {\em Numer. Math. 51, no. 2, 123–141.}, 1987.

\bibitem[Roc70]{Rockfellar}
R.T. Rockfellar.
\newblock {\em {Convex Analysis}}.
\newblock Princeton University Press, 1970.

\bibitem[RSN90]{Riesz_Nagy_1990}
Frigyes Riesz and B{\'e}la Sz.-Nagy.
\newblock {\em Functional analysis}.
\newblock Dover Books on Advanced Mathematics. Dover Publications Inc., New
  York, 1990.
\newblock Translated from the second French edition by Leo F. Boron, Reprint of
  the 1955 original.

\bibitem[SS10]{Simoncini_Szyld_2010}
Valeria Simoncini and Daniel~B. Szyld.
\newblock {On the field of values of oblique projections}.
\newblock {\em Linear Algebra Appl. 433 (2010), no. 4, 810–818.}, 2010.

\bibitem[TUZ03]{Takemoto_Uchiyama_Zsido_2003}
Hideo Takemoto, Atsushi Uchiyama, and Laszlo Zsido.
\newblock The {$\sigma$}-convexity of all bounded convex sets in {$\Bbb R^n$}
  and {$\Bbb C^n$}.
\newblock {\em Nihonkai Math. J.}, 14(1):61--64, 2003.

\end{thebibliography}


\end{document}